\documentclass[11pt]{amsart}
\usepackage{hyperref}
\usepackage{multirow,bigdelim}
\usepackage{amsfonts}
\usepackage{dsfont}
\usepackage{amsmath}
\usepackage{mathrsfs}

\DeclareSymbolFont{SY}{U}{psy}{m}{n}
\DeclareMathSymbol{\emptyset}{\mathord}{SY}{'306}

\theoremstyle{plain}

\newtheorem{thm}{Theorem}[section]
\newtheorem{cor}[thm]{Corollary}
\newtheorem{lem}[thm]{Lemma}
\newtheorem{prop}[thm]{Proposition}
\theoremstyle{definition}
\newtheorem{defn}[thm]{Definition}
\newtheorem{rem}[thm]{Remark}
\newtheorem{ex}[thm]{Example}

\numberwithin{equation}{section}

\def\C{{\mathbb C}}

\def\beq{\begin{eqnarray}}
\def\eeq{\end{eqnarray}}
\def\beqa{\begin{eqnarray*}}
\def\eeqa{\end{eqnarray*}}

\allowdisplaybreaks[0]

\begin{document}

\title[Curvature formulaes of  holomorphic curve on $C*$-algebras and Cowen-Douglas Operators]{Curvature formulaes of holomorphic curves on $C^*$-algebras and Cowen-Douglas Operators}

\author{ Kui Ji}



\curraddr{Department of Mathematics, Hebei Normal University,
Shijiazhuang, Hebei 050016, China} \email{jikui@hebtu.edu.cn}
\thanks{The author was supported Supported by National Natural
Science Foundation of China (Grant No. 11471094) and Hebei Natural Science Foundation for Distinguished Young Scholars (Grant No. A201605219) }.





\begin{abstract}
For $\Omega\subseteq \mathbb{C}$  a connected open set, and
${\mathcal U}$  a unital $C^*$-algebra, let  ${\mathcal I}
({\mathcal U})$ and ${\mathcal P}({\mathcal U})$  denote the sets of
all idempotents and projections in ${\mathcal U}$ respectively.
${\mathcal P}({\mathcal U})$ is called as the Grassmann manifold of
$\mathcal U$ and ${\mathcal I} ({\mathcal U})$ is called as the
extended Grassmann manifold. If $P:\Omega \rightarrow {\mathcal
P}({\mathcal U})$ is a real-analytic ${\mathcal U}$-valued map which
satisfies $\overline{\partial} PP=0$, then $P$ is called
  a  holomorphic curve on  ${\mathcal P}({\mathcal U})$.
   In this note, we will define the
formulaes  of curvature and it's covariant derivatives for
holomorphic curves on $C^*$-algebras. It can be regarded as the
generalization of curvature and it's covariant derivatives of the
classical holomorphic curves.   By using the curvature formulae, we
give the unitarily and similarity  classifications for the holomorphic  curves and extended
holomorphic curves on $C^*$-algebras respectively.  And we also give a
description of the trace of the covariant derivatives of curvature for any
Hermitian holomorphic vector bundles. As applications, we also
discuss the relationship between holomorphic curves, extended
holomorphic curves, similarity of holomorphic Hermitian vector
bundles and similarity of Cowen-Douglas operators.
\end{abstract}

\maketitle

\section{Introduction}

In this note, we will give the  formulaes of curvature and it's
covariant derivatives for  holomorphic curves in  Grassmann
manifolds in a $C^*$-algebraic setting and give the unitarily and similarity  classifications for the holomorphic  curves and extended
holomorphic curves on $C^*$-algebras respectively. 

In Cowen-Douglas theory, holomorphic curve in Grassmann manifold is
a basic and important concept. Let $\mathcal H$ be a complex
separable Hilbert space and $Gr(n,{\mathcal H})$ denote
$n$-dimensional Grassmann manifold, the set of all $n$-dimensional
subspaces of ${\mathcal H}$.  A map $f: \Omega \rightarrow
Gr(n,{\mathcal H})$ is called as a holomorphic curve, if there exist
$n$ holomorphic ${\mathcal H}$-valued functions $ e_1 ,e_2 ,\ldots,
e_n$ on ${\Omega}$ such that $f(\lambda)=\bigvee
\{e_1(\lambda),\ldots, e_n(\lambda)\}$ for each $\lambda\in \Omega$,
where symbol ``$\bigvee$'' denotes the closure of linear span (See
\cite{CD}).    And the concept of the  holomorphic curve on
$C^*$-algebras was first introduced by M. Martin and N. Salinas in
\cite{MS1}.  Let ${\mathcal U}$ be a
unital
  $C^*$-algebra, then $p\in {\mathcal U}$ is called a projection in ${\mathcal U}$ whenever
 $p^2=p=p^*$, and
 $ {\mathcal P}({\mathcal U})$ denote the set of all projections in ${\mathcal U}$ which is called Grassmann manifold of ${\mathcal U}$.
 Let $\Omega\subseteq \mathbb{C}$ be a connected open set. If $P:\Omega \rightarrow   {\mathcal P}({\mathcal
 U})$ is a real-analytic ${\mathcal U}$-valued map, then it is called an holomorphic curve on ${\mathcal P}({\mathcal U})$
 ( in order to discriminate ordinary holomorphic curve). Let ${\mathcal I}(
{\mathcal U})$ denote the set of all of the idempotents in
${\mathcal U}$.  ${\mathcal I}({\mathcal U})$  is called as the
extended Grassmann manifold.  We say a real-analytic map ${\mathcal I}:\Omega\rightarrow {\mathcal
I}({\mathcal U})$ as an extended holomorphic curve if the following
statements hold $$ \overline{\partial} {\mathcal I}(\lambda)={\mathcal
I}(\lambda)\overline{\partial} {\mathcal I}(\lambda), \partial{\mathcal I}(\lambda)=\partial{\mathcal I}(\lambda){\mathcal
I}(\lambda),\overline{\partial} {\mathcal I}(\lambda){\mathcal
I}(\lambda)=0,{\mathcal I}(\lambda){\partial} {\mathcal
I}(\lambda)=0, \forall \lambda \in \Omega.$$

This class of holomorphic curves in a $C^*$-algebraic setting has
been studied by C. Apostol, M. Martin, N. Salinas and D. R. Wilkins
in a number of articles\cite{AM,M,MS1,MS2,MS3,S,W}.  In 1981, a
$C^*$-algebra approach to Cowen-Dougals theory was given by C.
Apostol and M. Martin (cf. \cite{AM}).  And M. Martin and N. Salinas
did a series work of holomorphic curves on extended flag manifolds
and extended Grassmann manifolds (\cite{M,MS1,MS2,MS3,S}). This
kind of researches on  holomorphic curves can be regarded as one of
generalization of Cowen-Douglas theory on $C^*$-algebras.

  In the paper \cite{CD}, M. J. Cowen and R. G. Douglas introduced a
class of operators related to complex geometry now referred to as
Cowen-Douglas operators (See Example \ref{exinfinite}). There exists
a natural connection between  holomorphic curves and this class of
operators. For $\mathcal{H}$  a complex and separable Hilbert space,
let $\mathcal{L}(\mathcal{H})$ be the set of bounded linear
operators on $\mathcal{H}$. Let $\Omega$ be a open connected subset
of complex plane $\mathbb{C}$. A class of Cowen-Douglas operator
with index $n$: $B_n(\Omega)$ is defined as follows \cite{CD}:
$$\begin{array}{lll}B_n(\Omega)=:\{T\in \mathcal{L}(\mathcal{H}):
&(i)\,\,\Omega\subset \sigma(T)=:\{\lambda\in \mathbb{C}:T-\lambda
I~~
\mbox{is not invertible}\},\\
&(ii)\,\,\bigvee_{\lambda\in \Omega}\mbox{Ker}(T-\lambda)=\mathcal{H},\\
&(iii)\,\,\mbox{Ran}(T-\lambda)=\mathcal{H},\\
&(iv)\,\,\mbox{dim Ker}(T-\lambda)=n, \forall\lambda\in\Omega.\}
\end{array}$$
For any operator $T\in B_n(\Omega)$, it is shown that one can find a
holomorphic family of eigenvectors $\{e_i(\lambda), \lambda\in
\Omega\}^n_{i=1}$ such that $Te_i(\lambda)=\lambda e_i(\lambda),$
for any $\lambda\in \Omega$. A holomorphic curve with $n$ dimension
is a map from $\mathcal{H}$ to Grassmann manifold
$Gr(n,\mathcal{H})$ defined as $f(\lambda)=:\bigvee
\{e_i(\lambda),i=1,2,\cdots, n\}$ for $\lambda\in \Omega$.

 M. J. Cowen and R. G. Douglas obtained a unitarily 
classification theorem of holomorphic curves in \cite{CD}. It is proved that
 curvature function and it's derivative are unitarily invariants of the
holomorphic curves and Cowen-Doulgas operators by means of complex
hermitian geometry techniques.

 For any
Cowen-Douglas operator $T$, there exists a Hermitian holomorphic
bundle $E_T$ with the fiber $f(\lambda), \lambda\in\Omega$.  We call
two linear bounded operators $T$ and $S$ are unitarily equivalent if
and only if there exists a unitary operator $U\in
\mathcal{L}(\mathcal{H})$ such that $T=USU^*$ (denoted by $T\sim_{u}
S$). For two holomorphic curves $F$ and $G$ defined on $\Omega$, if
there exists a unitary operator $U\in \mathcal{L}(\mathcal{H})$ such
that $f(\lambda)=Ug(\lambda),\forall \lambda\in \Omega$, then we
call them are unitarily equivalent (denoted by $f\sim_u g$.

In \cite{CD}, it is shown that unitary equivalence of operator $T$
can be deduced to the same problem of holomorphic curve $F$
associate to it. Following M. I. Cowen and R. G. Doulgas \cite{CD},
a curvature function for $T\in B_n(\Omega)$ can be defined as:

$$K_T(\lambda)
=-\frac{\partial}{\partial \overline{\lambda}}(h^{-1}\frac{\partial
h}{\partial \lambda}), \mbox{for all}~ \lambda\in \Omega,$$

where the metric
$$h(\lambda)=(<e_j(\lambda),e_i(\lambda)>)_{n\times n},\forall
\lambda\in \Omega,$$ and
$\{e_1(\lambda),e_2(\lambda),\cdots,e_n(\lambda)\} $ are the frames
of $E_T$.

 Let $E_T$ be a
Hermitian holomorphic bundle induced by a Cowen-Douglas operator
$T$, and $K_T$ be a curvature of $T$. Then covariant  partial
derivatives of curvature $K_T$ are defined as the following:

 (1)\,\,$K_{T,\overline{z}}=\frac{\partial}{\overline{\partial}\lambda}(K_{T});$

 (2)\,\,$K_{T,z}=\frac{\partial}{\partial \lambda}(K_T)+[h^{-1}\frac{\partial}{\partial \lambda}h,K_{T}].$

 A
remarkable result is also proved in \cite{CD}: For $T,S \in
B_1(\Omega)$, $T\sim_u S $ if and only if $K_T=K_S$ on $\Omega$.

Let  $T_1,T_2\in B_n(\Omega)$. Then $T_1\sim_{u} T_2$ if and only
if there exists an isometry $V: E_{T_1}\rightarrow E_{T_2}$ such
that
$$VK_{T_1,z^i\overline{z}^j}=K_{T_2,z^i\overline{z}^j}V, i,j=0,1,\cdots,n-1.$$

Subsequently, the curvature function turns into an important object
of the research of
 Cowen-Douglas operators. R. G. Douglas,A. Kor\'{a}nyi, G. Misra, K. Guo, D. N. Clark, M. Uchiyama, H. Kwon, S. Treil,  L. Chen, and S. S. Roy
  and many other mathematicians
 did a lot of work around the curvature  (cf.\cite{CDG,CM1,CM2,CM3,CD,Kwon2,KM1,KM2,Dinesh,Kwon1,Misra1,Misra2,Misra3,Sarkar,U} ). On the
other hand, by using $K_0$-group, C. Jiang, X. Guo and K. Ji
 concerned the problems of similarity classification of Cowen-Douglas operators and some
holomorphic curves(cf. \cite{Jikui,Jiang,JGJ,JJ}).

Same to the progress of researches of Cowen-Douglas operators, we also start
from the unitarily equivalence of this kind of holomorphic curves.
Let $P,Q:\Omega\rightarrow P({\mathcal U})$ be two
holomorphic curves on same C*-algerba. We say that $P$ and $Q$ are unitarily equivalent
(denoted by $P\stackrel{u}{\sim}Q$) if there exists a fixed unitary
$U\in {\mathcal U}$ such that $P(\lambda)=UQ(\lambda)U^{*},\forall
\lambda\in\Omega$ (\cite{MS1}).

In \cite{MS1}, M. Martin and N. Salinas give the unitarily
invariants of extended holomorphic curve $P$  by considering the
partial derivatives $\partial^IP\overline{\partial}^JP$, $I,J\in
\mathds{N}$.  As we mentioned above, the important and interesting
part of the researches in holomorphic curves and Cowen-Douglas
operators is the intrinsic connection with complex geometry. One can
decide the unitarily equivalence of two operators by calculating
their curvatures and it's covariant derivatives. From this view
point, we also need to search for the geometry unitarily invariants of
holomorphic curve on $C^*$-algebras. So a natural question is the
following :

{\bf Question }\,\, {\it What is the curvature and it's covariant
derivatives for the  holomorphic curves on $C^*$-algebras? Are they also the unitarily invariants of  holomorphic curves on
$C^*$-algebras?}

 To answer this question,
we want to characterize the curvature and it's covariant partial
derivatives's formulaes and unitary equivalence problem of 
holomorphic curves in $C^*$-algebras with these geometry concepts.

On the other hand, people also consider the  similarity
classification of holomorphic curves and Cowen-Douglas operators. As
we all known, curvature is not the similarity invariant for holomorphic in the classical case(cf
\cite{CM1,CM2}). In \cite{JJ}, we give a similarity classification
of holomorphic curves involving the $K_0$ group of the holomorphic
curve's commutant algebra. But we still have no any geometry
invariants for the similarity of holomorphic curves and
Cowen-Douglas operators.

In \cite{Kwon1},by considering the Hilbert-Schmidt norm of the
partial derivative of analytic projection (or the trace of the
curvature of corresponding operator (See details in \cite{HJK}), H.Kwon and S.Treil
characterize contractions with certain property that are similar to
the backward shift  in the Hardy space. This analytic projection is
also a kind of holomorphic curve on $C^*$-algebras. And this result
was also generalized to the weighted Bergman shift case by R. G.
Douglas, H. Kwon and S. Treil (cf \cite{Kwon2}). 

In this paper, we will give a similarity classification  of  extended holomorphic curves on $C^*$-algebras
  by using
this curvature formulaes. As an application, we also describe the trace of derivatives of curvatures for
Cowen-Dougals operators in the form of holomorphic curves on
$C^*$-algebras.

 The paper is organized as follows. In \S1 some notations and known results will be
introduced.  In \S2, we define the curvature and it's covariant
derivatives of the holomorphic curves and extended holomorphic curves and we also give 
unitarily classification and similarity classification theorem of holomorphic curves and extended holomorphic  curves on
$C^*$-algebras by using these curvature and covariant derivatives. In
\S3, we will discuss the relationship between curvature formulae of 
extended holomorphic curves and similarity of Cowen-Douglas operators.
In \S4, We will discuss the applications.

We will introduce some notations and results first, and all the
notations are adopted from \cite{CD}, \cite{Jikui} and \cite{MS1}.

To simplify the notation,    we use the  symbol
``$\overline{\partial}^J\partial^{I}$''  denotes  partial derivative
`` $\frac{\partial^{J+I}}{\partial^J
\overline{\lambda}\partial^{I}\lambda} $'', where $I,J$ are
non-negative integers. And for any $I$ and $J$,

(1)\,\,symbol $\overline{\partial}^J$ stands for
$\overline{\partial}^J\partial^0$ and  $\partial^I$ stands for
$\overline{\partial}^0\partial^I$,

(2)\,\,symbol  $\partial$ stands for $\partial^1$, and
$\overline{\partial}$ stands for $\overline{\partial}^1$,

(3)\,\, $\overline{\partial}^J\partial^{I}P=P$, when $J=I=0$.

Firstly, we  need a criterion for determining the holomorphic map
from $\Omega$ to ${\mathcal P({\mathcal U})}$.

\subsection{}\label{defn1} \cite{MS1}\,\, Let ${\mathcal U}$ be a unital $C^{*}$-algebra.
Let $P:\Omega\rightarrow {\mathcal P({\mathcal U})}$ be a ${\mathcal
U}$-valued infinitely differentiable map. Then $P$ is called
holomorphic if and only if  \begin{align*}\overline{\partial}
P(\lambda)=P(\lambda)\overline{\partial} P(\lambda), \forall \lambda
\in \Omega.
 \tag{1.1}\label{1.1} \end{align*}

Since $P(\lambda)$ is a projection, for any $\lambda\in \Omega$,
we can get that
$$\overline{\partial} P(\lambda)=[\overline{\partial} P(\lambda)]P(\lambda)+P(\lambda)[\overline{\partial} P(\lambda)].$$
So \ref{1.1} is equivalent to say that
$$[\overline{\partial}P(\lambda)]P(\lambda) =0\Longleftrightarrow
\partial P(\lambda)=[\partial
P(\lambda)]P(\lambda)\Longleftrightarrow P(\lambda)\partial
P(\lambda)=0, \forall \lambda\in \Omega.$$  By a direct computation,
we also have
 \begin{align*}\overline{\partial}\partial^{J}P=\partial^{J}P
\overline{\partial}P-\overline{\partial}P\partial^{J}P
-\sum\limits_{k=1}^{J-1}C^k_J(\partial^{J-K}P\overline{\partial}P\partial^k
P), \forall J\in \mathbb{Z}^+. \tag{1.2}\label{1.2} \end{align*}

 \begin{align*}\overline{\partial}^{I}\partial P=\partial P
\overline{\partial}^IP-\overline{\partial}^IP\partial P
-\sum\limits_{k=1}^{I-1}C^k_I(\overline{\partial}^{I-K}P
\partial P\overline{\partial}^k P), \forall I\in \mathbb{Z}^+. \tag{1.3}\label{1.3} \end{align*}
and
$$\overline{\partial}^JPP=P\partial^IP=0,\forall I, J\in
\mathbb{N}.$$ For the general case, each derivative
$\overline{\partial}^J\partial^IP, I, J\in \mathbb{N}$ may be
expressed as a sum of monnomials of the form (See more details in
\cite{MS1})
$$\pm
[\overline{\partial}^{I_1}P][\partial^{J_1}P]\cdots[\overline{\partial}^{I_k}P][\partial^{J_k}P]$$
and
 \begin{align*}\pm
[\partial^{J_1}P][\overline{\partial}^{I_1}P]\cdots[\partial^{J_k}P][\overline{\partial}^{I_k}P]
\tag{1.4}\label{1.4}
\end{align*}

\subsection{}
 Let ${\mathcal U}$ be a unital $C^{*}$-algebra, and
$P:\Omega\rightarrow {\mathcal P({\mathcal U})}$ be an 
holomorphic curve.
 For each $\lambda\in \Omega$ and every $\tau\in \mathbb{Z}_{+}\cup \{\infty\}$,
 set
 $${\mathcal B}^{\tau}_{\lambda}=\{\overline{\partial}^{J}P(\lambda)\partial^{I}P(\lambda):I,J\in \mathbb{Z}_{+}, I,J\leq \tau\}.$$
 Let
${\mathcal U}^{\tau}_{\lambda}$ be the closure of $*$-subalgebra of
$\mathcal U$ generated by ${\mathcal B}^{\tau}_{\lambda}$ with the
following property:
 $$ {\mathcal U}^{0}_{\lambda}\subseteq {\mathcal U}^{1}_{\lambda}\subseteq\cdots \subseteq {\mathcal U}^{\infty}_{\lambda}.$$

By using notations mentioned above, M. Martin and  N. Salinas
defined  a substitute in $C^{*}$-algebra for  Cowen-Douglas class
$B_n(\Omega)$:

\begin{defn}\cite{MS1}\label{Uinfi} Let $k\geq 1$ be an integer. If the following
conditions are satisfied, then  holomorphic curve
$P:\Omega\rightarrow P({\mathcal U})$ is said to be in the class
${\mathcal A}_{k}(\Omega,{\mathcal U}):$

$(1)$\,\, For each $\lambda\in \Omega$, ${\mathcal
U}^{\infty}_{\lambda}$ is a finite-dimensional $C^*$-algebra.

$(2)$\,\,If $k_{\lambda}$ denotes the cardinal of any maximal
collection of mutually orthogonal minimal projections in ${\mathcal
U}^{\infty}_{\lambda}$, then $$k_{\lambda}\leq k.$$

$(3)$\,\,If $a\in {\mathcal U}$ and $aP(\lambda)=0$ for every
$\lambda\in\Omega$, then $a=0$.
\end{defn}

\begin{defn}\cite{MS1} Let $\lambda\in\Omega$ and $\alpha\in \mathbb{Z}_{+}$
be a fixed integer. We say that $P$ and $Q$ {\bf have order of
contact $\alpha$ at $\lambda$ }if there exists a unitary $\nu$ such
that
$$\nu\overline{\partial}^{J}P(\lambda)\partial^{I}P(\lambda)\nu^{*}=\overline{\partial}^{J}Q(\lambda)\partial^{I}Q(\lambda),~
\forall 0\leq I,J\leq \alpha, \eqno(1.5)\label{1.5}$$
\end{defn}

We say $\mathfrak{G}\subset {\mathcal U}$ is a {\bf separating
subset} of ${\mathcal U}$, if $\{a\in {\mathcal U}:as=0, s\in
\mathfrak{G} \}=\{0\}.$ Assume $\mathfrak{G}$, $\mathfrak{T}$ are
two separating subsets of ${\mathcal U}$,
$\theta:\mathfrak{G}\rightarrow \mathfrak{T}$ is a given bijection.

M.Martin and  N.Salinas proved the following related rigidity
theorem for ${\mathcal A}_k(\Omega,{\mathcal U})$ class on
$C^{*}$-algebra.

\begin{thm}\label{MSlemma}[Theorem 4.5, \cite{MS1}] Let ${\mathcal U}$ be a unital C*-algebra. Suppose that 
holomorphic curves $P,Q:\Omega\rightarrow {\mathcal P({\mathcal
U})}$ belong to the class  ${\mathcal A}_k(\Omega,{\mathcal U})$.
Then the following two statements are equivalent:

$(1)$\,\,$P$ and $Q$ are unitarily equivalent;

$(2)$\,\,$P$ and $Q$ have order of contact $\alpha$ at each
$\lambda\in \Omega$.
\end{thm}

\section{Curvature formulae of the  holomorphic curves and extended holomorphic curves on C*-algebras}

\subsection{Holmorphic curves}

\begin{defn}\label{holomorphic}
Let $\Omega$ be a connected open subset of $\mathbb{C}$ and
${\mathcal U}$ be a unital $C^*$-algebra.  Let $P:\Omega\rightarrow
{\mathcal P({\mathcal U})}$ be an  holomorphic curve.

We say ${\mathscr K}_{i,j}(P), 0\leq i,j$ defined as the following
to be the curvature and covariant derivatives of curvature of the
 holomorphic curve $P$, where
 \begin{align*}  {\mathscr K}(P):={\mathscr K}_{0,0}(P)=\overline{\partial}P{\partial P}, {\mathscr K}_{i+1,j}(P)=P(\partial ({\mathscr K}_{i,j}(P))),
{\mathscr K}_{i,j+1}(P)=(\overline{\partial}({\mathscr
K}_{i,j}(P)))P, \forall i,j\in \mathbb{Z}_{+}. \end{align*}
\end{defn}

Let  $B$ be a unital $C^*$-algebra. A Hilbert $B$-module
$l^2(\mathbb{N},B)$ is defined as $$l^2(\mathbb{N},B)=:\{(a_i)_{i\in
\mathbb{N}}:a_i\in B,\forall i\in \mathbb{N},
\mbox{and}\sum\limits_{i\in \mathbb{N}}||a_i||^2<\infty\}.$$ We
denote the set of all the linear bounded operators on
$l^2(\mathbb{N},B)$ by ${\mathscr L}(l^2(\mathbb{N},B))$. Then
${\mathscr L}(l^2(\mathbb{N},B))$ is a $C^*$-algebra.

 By using the notations
introduced above, we will give the following  holomorphic
curve class which we concern in this paper.

\begin{defn}\label{Pn}
Let $B$ be a unital  $C^*$-algebra. For ${\mathcal U}={\mathscr L}(l^2(\mathbb{N},B))$.
Let 
${\mathcal P}_n(\Omega ,{\mathcal U})$ denotes the orthogonal 
projection valued functions $P:\Omega\rightarrow {\mathcal U}$ which
satisfies:  $$RanP(\lambda)=\bigvee\{\alpha_i(\lambda),i=1,2,\cdots,n\}. $$ where 
 $\alpha_i:\Omega \rightarrow 
l^2(\mathbb{N},B), i=1,2\cdots,n$ are holomorphic functions.

Define $\alpha : \Omega\rightarrow 
{\mathscr L}(\mathbb{C}^n, l^2(\mathbb{N},B))$ as follows 
$$\alpha(\lambda)(w_1,w_2,\cdots,w_n)=\sum\limits_{i=1}^nw_i\alpha_i(\lambda), \forall w_i\in {\mathbb C},  \lambda\in \Omega.$$
Then
$$P(\lambda)=\alpha(\lambda)\cdot(\alpha^{*}(\lambda)\cdot\alpha(\lambda))^{-1}\cdot\alpha^{*}(\lambda),
\forall \lambda \in \Omega,$$ 

\end{defn}

 \begin{prop}\label{hlc}
Let $B$ be a unital $C^*$-algebra and ${\mathcal U}={\mathscr
L}(l^2(\mathbb{N},B))$.  Then the $C^{\infty}$ map
$P:=\alpha\cdot(\alpha^{*}\cdot\alpha)^{-1}\cdot\alpha^{*}$ defined
in Definition \ref{Pn} is an  holomorphic curve.
\end{prop}

\begin{proof}

Firstly, we have that $P(\lambda)$ is orthogonal projection for any
$\lambda\in \Omega.$

 By Definitions in \ref{defn1},  we only need to prove that
 $P$ satisfies the formulae \ref{1.1}. Note that $\overline{\partial}\alpha=0$ and $h=\alpha^{*}\cdot\alpha$.
 Then
$$\begin{array}{llll} \overline{\partial}P
P&=&\overline{\partial}(\alpha(\alpha^{*}\cdot\alpha)^{-1}\cdot\alpha^{*})(\alpha\cdot(\alpha^{*}\cdot\alpha)^{-1}\cdot\alpha^{*})\\
&=&\overline{\partial}(\alpha h^{-1}\alpha^{*})(\alpha h^{-1}\alpha^{*})\\
&=&(\alpha \overline{\partial}( h^{-1})\alpha^{*}+ \alpha h^{-1}\overline{\partial}\alpha^{*})(\alpha h^{-1}\alpha^{*})\\
&=&(\alpha \overline{\partial}( h^{-1})\alpha^{*}\cdot\alpha h^{-1}\alpha^{*}+ \alpha h^{-1}\overline{\partial}\alpha^{*}\cdot\alpha h^{-1}\alpha^{*}\\
&=&0.
\end{array}$$ This finishes the proof of the Proposition.
\end{proof}

\begin{ex}\label{exinfinite}\cite{Jikui}  Let $E(\lambda),\lambda\in \mathbb{D}$ be an analytic family of
subspaces of Hilbert space ${\mathcal H}$ (or holomorphic curve).
And let $P(\lambda)$ be the orthogonal projection onto $E(\lambda)$.
Then $P: \mathbb{D} \rightarrow P({\mathcal L}({\mathcal H}))$ is an
 holomorphic curve (cf \cite{MS1}). As we all known, the
subspace $E(\lambda)$ is equal to the range of $F(\lambda)$ where
$F$ is a left invertible analytic operator-valued function. And
$$P=F(F^*F)^{-1}F^*.$$
In the example above, when we assume ${\mathcal U}={\mathcal
L}({\mathcal H})$, we can see that $\{\alpha_i(\lambda)\}^{n}_{i=1}$
are the frames of $E(\lambda)=Ran F(\lambda)$ for any $\lambda\in
\mathbb{D}$.

 For the finite dimension case, let ${\mathcal U}$ be
$M_2(\mathbb{C})$ and $\Omega\subseteq \mathbb{C}$ be a connected
open set, and let $P:\Omega\rightarrow M_2(\mathbb{C})$ defined by
$$P(\lambda)=\dfrac{1}{1+|\lambda|^2}\left(\begin{array}{cccc}
1&\overline{\lambda}\\
\lambda&|\lambda|^2
\end{array}\right), \forall \lambda\in \Omega.$$
Then $P$ is called Bott projection in algebra K-theory. When we
assume that $\alpha(\lambda):=(1,\lambda)^T\in \mathbb{C}^2$ and
$\alpha^*(\lambda):=(1,\overline{\lambda})$, then we have
 $$P(\lambda)=\alpha(\lambda)\cdot(\alpha^{*}(\lambda)\cdot\alpha(\lambda))^{-1}\cdot\alpha^{*}(\lambda), \forall \lambda\in \Omega,$$
where $\alpha^{*}(\lambda)\cdot\alpha(\lambda)=1+|\lambda|^2$. And
$P$ is an  holomorphic curve on $\Omega$.

\end{ex}

\begin{defn}\label{excurvature} Let  $P\in {\mathcal P}_n(\Omega, {\mathcal U})$. Considering
$l^2(\mathbb{N},B)$ is a  Hilbert C*-module, denote the metric
 $$h(\lambda)=<\alpha(\lambda),\alpha(\lambda)>=\alpha^{*}(\lambda)\cdot\alpha(\lambda).$$
  An curvature function of $P$ is defined as
$$K_P=-\frac{\partial}{\partial \overline{\lambda}}(h^{-1}\frac{\partial h}{\partial \lambda}), \mbox{for all}~ \lambda\in \Omega.$$
And the partial derivatives of curvature are defined as the
following:

 (1)\,\,$K_{P,\overline{\lambda}}=\frac{\partial}{\overline{\partial}\lambda}(K_{P});$

 (2)\,\,$K_{P,\lambda}=\frac{\partial}{\partial \lambda}(K_P)+[h^{-1}\frac{\partial}{\partial \lambda}h,K_{P}],~\mbox{for all}~ \lambda\in \Omega. $

By the definition above,  we can get the partial derivatives of
curvature: $K_{P,\lambda^i\overline{\lambda}^j}$, $i,j\in
\mathbb{N}\cup \{0\}$ by using the inductive formulaes above. And
this curvature function and the partial derivatives of curvature
function are same to the curvature of Cowen-Douglas operator in
form.
\end{defn}

In the following parts of this paper, for the sake of simplicity of
the notations, we will also use $\alpha \alpha^*$ and $\alpha^*
\alpha$ instead of $\alpha \cdot \alpha^*$ and $\alpha^* \cdot
\alpha$ respectively.

\begin{lem}\label{main} Let $P\in {\mathcal P}_n(\Omega, {\mathcal
 U})$. And there exist $\alpha : \Omega\rightarrow 
{\mathscr L}(\mathbb{C}^n, l^2(\mathbb{N},B))$ such that
$$P(\lambda)=\alpha(\lambda)\cdot(\alpha^{*}(\lambda)\cdot\alpha(\lambda))^{-1}\cdot\alpha^{*}(\lambda).$$
Then ${\mathscr K}_{i,j}(P):\Omega\rightarrow {\mathcal U},$ $
i,j=0,1,\cdots,n$ defined in Definition 2.1
 satisfy  the following conclusion:
$${\mathscr K}_{i,j}(P)(\lambda)=\alpha(\lambda)(-K_{P,z^i,{\overline
 z}^j})h^{-1}\alpha^*(\lambda), \forall \lambda\in \Omega,$$
 where $h=\alpha^*\cdot \alpha$.

\end{lem}

\begin{proof}

Let
$P(\lambda)=\alpha(\lambda)\cdot(\alpha^{*}(\lambda)\cdot\alpha(\lambda))^{-1}\cdot\alpha^{*}(\lambda)=\alpha(\lambda)\cdot
h^{-1}(\lambda)\cdot\alpha^{*}(\lambda), \forall \lambda\in \Omega.$
And set $h(\lambda)=\alpha^{*}(\lambda)\cdot\alpha(\lambda).$

 Then we have the following
claim:

{\bf Claim 1}\,\,
\begin{align*}\partial^I P=(\partial^I \alpha
h^{-1} +C^1_{I}\partial^{I-1}\alpha\partial
h^{-1}+\cdots+\C^{k}_{I}\partial^{I-k}\alpha \partial^k
h^{-1}+\cdots+\alpha \partial^{I}h^{-1})\alpha^*, \forall I\in
\mathds{N}, \tag{2.3}\label{2.3}
\end{align*}

\begin{align*} \overline{\partial}^J P=\alpha(\overline{\partial}^J  h^{-1}\alpha^*
+C^1_{J}\overline{\partial}^{J-1}
h^{-1}\overline{\partial}\alpha^*+\cdots+C^k_{J}\overline{\partial}^{J-k}
h^{-1}\overline{\partial}^k\alpha^*+\cdots+
h^{-1}\overline{\partial}^J\alpha^*), \forall J\in \mathds{N}.
\tag{2.4}\label{2.4}
\end{align*}

 Since $(\partial^I P)^*=\overline{\partial}^I P, \forall
I\in \mathds{N},$ then we only need to prove the formulae \ref{2.4}.
When $I=1$, note that $\overline{\partial}\alpha=0$, we have that
$$\begin{array}{llll}\overline{\partial} P&=&\overline{\partial} (\alpha\cdot
h^{-1}\cdot\alpha^{*})\\
&=&\alpha \overline{\partial}h^{-1}\alpha^*+\alpha
h^{-1}\overline{\partial}\alpha^*. \end{array}$$

By induction proof, suppose the following formulaes hold:
$$\overline{\partial}^{J-1} P=\alpha(\overline{\partial}^{J-1}  h^{-1}\alpha^*
+C^1_{J-1}\overline{\partial}^{J-2}
h^{-1}\overline{\partial}\alpha^*+\cdots+C^k_{J-1}\overline{\partial}^{J-k-1}
h^{-1}\overline{\partial}^k\alpha^*+\cdots+
h^{-1}\overline{\partial}^{J-1}\alpha^*).$$ Then we have
$$\begin{array}{llll}\overline{\partial}( \overline{\partial}^{J-1} P)&=&
\overline{\partial}(\alpha(\overline{\partial}^{J-1}  h^{-1}\alpha^*
+\cdots+C^k_{J-1}\overline{\partial}^{J-k-1}
h^{-1}\overline{\partial}^k\alpha^*+\cdots+
h^{-1}\overline{\partial}^{J-1}\alpha^*) )\\
&=&\alpha(\overline{\partial}^{J}  h^{-1}\alpha^*
+\overline{\partial}^{J-1}
h^{-1}\overline{\partial}\alpha^*+\cdots+\overline{\partial}(C^k_{J-1}\overline{\partial}^{J-k-1}
h^{-1}\overline{\partial}^k\alpha^*)+\cdots\\
&+&
\overline{\partial}h^{-1}\overline{\partial}^{J-1}\alpha^*+h^{-1}\overline{\partial}^{J}\alpha^*).
\end{array}$$
Note that $$
\overline{\partial}(C^k_{J-1}\overline{\partial}^{J-k-1}
h^{-1}\overline{\partial}^k\alpha^*)=C^k_{J-1}\overline{\partial}^{J-k}
h^{-1}\overline{\partial}^k\alpha^*+C^k_{J-1}\overline{\partial}^{J-k-1}
h^{-1}\overline{\partial}^{k+1}\alpha^*,$$ $$
\overline{\partial}(C^{k+1}_{J-1}\overline{\partial}^{J-k-2}
h^{-1}\overline{\partial}^{k+1}\alpha^*)=C^{k+1}_{J-1}\overline{\partial}^{J-k-1}
h^{-1}\overline{\partial}^{k+1}\alpha^*+C^{k+1}_{J-1}\overline{\partial}^{J-k-2}
h^{-1}\overline{\partial}^{k+2}\alpha^*$$ and
$$ C^k_{J-1}\overline{\partial}^{J-k-1}
h^{-1}\overline{\partial}^{k+1}\alpha^*+C^{k+1}_{J-1}\overline{\partial}^{J-k-1}
h^{-1}\overline{\partial}^{k+1}\alpha^*=C^{k+1}_{J}\overline{\partial}^{J-k-1}
h^{-1}\overline{\partial}^{k+1}\alpha^*$$

Then we have $$\overline{\partial}^J P=\alpha(\overline{\partial}^J
h^{-1}\alpha^* +C^1_{J}\overline{\partial}^{J-1}
h^{-1}\overline{\partial}\alpha^*+\cdots+C^k_{J}\overline{\partial}^{J-k}
h^{-1}\overline{\partial}^k\alpha^*+\cdots+
h^{-1}\overline{\partial}^J\alpha^*), \forall J\in \mathds{N}.$$ So
we finish the proof of Claim 1.

{\bf Claim 2}\,\, \begin{align*}\overline{\partial} P \partial P
=\alpha(-K_P)h^{-1}\alpha^*;\tag{2.5}\label{2.5}\end{align*}
\begin{align*} \overline{\partial}
P
\partial^2 P=\alpha(-(K_P)_z)h^{-1}\alpha^*;\tag{2.6}\label{2.6}\end{align*}
\begin{align*} \overline{\partial}^2P \partial
P=\alpha(-(K_P)_{\overline{z}})h^{-1}\alpha^*;\tag{2.7}\label{2.7}\end{align*}
\begin{align*}\overline{\partial}^2 P
\partial^2 P-2(\overline{\partial} P \partial
P)^2=\alpha(-(K_P)_{z\overline{z}})h^{-1}\alpha^*;\tag{2.8}\label{2.8}\end{align*}

In fact, $$\begin{array}{llll} \overline{\partial} P\partial P
&=&(\alpha \overline{\partial}h^{-1}\alpha^*+\alpha
h^{-1}\overline{\partial}\alpha^*)(\partial \alpha h^{-1}\alpha^*
+\alpha\partial
h^{-1}\alpha^*)\\
&=&\alpha(\overline{\partial}h^{-1}\partial h
+h^{-1}\overline{\partial}{\partial}h)h^{-1}\alpha^*
\end{array}
$$
Since $K_P=-(\overline{\partial}h^{-1}\partial h
+h^{-1}\overline{\partial}{\partial}h),$ then we obtain the formulae
\ref{2.5}. And also note that  $\theta_P=h^{-1}\partial h,$ then we
have that
$$\begin{array}{llll}
\overline{\partial} P\partial^2P &=&(\alpha
\overline{\partial}h^{-1}\alpha^*+\alpha
h^{-1}\overline{\partial}\alpha^*)(\partial^2 \alpha h^{-1}\alpha^*
+2\partial^{1}\alpha\partial
h^{-1}\alpha^*+ \alpha \partial^2h^{-1}\alpha^*)\\
&=&\alpha(-\partial K_P+2\overline{\partial}h^{-1}\partial h\partial
h^{-1}h+2h^{-1}\partial\overline{\partial}h\partial
h^{-1}h+\overline{\partial}(h^{-1}\partial h h^{-1})\partial h-
\partial
h^{-1}\overline{\partial}\partial h)h^{-1}\alpha^*\\
&=&\alpha(-\partial K_P+2\overline{\partial}h^{-1}\partial h\partial
h^{-1}h+2h^{-1}\partial\overline{\partial}h\partial
h^{-1}h+\overline{\partial}h^{-1}\partial h h^{-1}\partial
h+h^{-1}\partial \overline{\partial}h h^{-1}\partial
h\\
&&+h^{-1}\partial h \overline{\partial}h^{-1}\partial h-
h^{-1}\partial h h^{-1}\overline{\partial}\partial h)h^{-1}\alpha^*\\
\end{array}
$$

Set $\partial h^{-1}h=-h^{-1}\partial h=-\theta_P$,  we have
$$\begin{array}{llll}
\overline{\partial} P\partial^2P &=&\alpha(-\partial
K_P+2\overline{\partial}h^{-1}\partial h\partial
h^{-1}h+2h^{-1}\partial\overline{\partial}h\partial
h^{-1}h+\overline{\partial}h^{-1}\partial h h^{-1}\partial
h+h^{-1}\partial \overline{\partial}h h^{-1}\partial
h\\
&&+h^{-1}\partial h \overline{\partial}h^{-1}\partial h+
h^{-1}\partial h h^{-1}\overline{\partial}\partial h)h^{-1}\alpha^*\\
&=&-\alpha(\partial K_P+(\overline{\partial}h^{-1}\partial h
+h^{-1}\partial\overline{\partial}h)\theta_P-\theta_P(
\overline{\partial}h^{-1}\partial h-
 h^{-1}\overline{\partial}\partial h))h^{-1}\alpha^*\\
 &=&-\alpha(\partial K_P+(\theta_PK_P-K_P\theta_P))h^{-1}\alpha^*\\
\end{array}
$$

By Definition \ref{excurvature}, we obtain
$$\begin{array}{llll}
(K_P)_z&=&\partial K_P+[\theta_P,K_P]\\
\end{array}$$

Then it follows that
$$\overline{\partial}
P
\partial^2 P=\alpha(-(K_P)_z)h^{-1}\alpha^*. $$

And $$\begin{array}{lllll} \overline{\partial}^2P\partial P&=&\alpha
\overline{\partial}(\overline{\partial}h^{-1}\alpha^*+h^{-1}\overline{\partial}\alpha^*)(\alpha\partial
h^{-1}+\alpha\partial h^{-1})\alpha^*\\
&=&\alpha
(\overline{\partial}^2h^{-1}\alpha^*+2\overline{\partial}h^{-1}\overline{\partial}\alpha^*+h^{-1}\overline{\partial}^2\alpha^*)(\alpha\partial
h^{-1}+\alpha\partial h^{-1})\alpha^*\\
&=&\alpha(-\overline{\partial} K_P +(\overline{\partial}^2 h^{-1} h+
2\overline{\partial}h^{-1}\overline{\partial}h+h^{-1}\overline{\partial}^2h)\partial
h^{-1}h)h^{-1}\alpha^*\\
&=&\alpha(-\overline{\partial} K
+\overline{\partial}(\overline{\partial}h^{-1}h+h^{-1}\overline{\partial}h)\partial
h^{-1}h)h^{-1}\alpha^*\\
&=&\alpha(-\overline{\partial} K_P
)h^{-1}\alpha^*\\
&=&\alpha (-K_{P}
)_{\overline{z}}h^{-1}\alpha^*.\\
\end{array}$$

Then it follows that
$$\overline{\partial}^2P\partial P=\alpha (-(K_{P}
)_{\overline{z}})h^{-1}\alpha^*. $$

And we also have \begin{align*}\begin{array}{llll}
\overline{\partial}^2P\partial^2 P&=&\alpha
(\overline{\partial}^2h^{-1}\alpha^*+2\overline{\partial}h^{-1}\overline{\partial}\alpha^*+h^{-1}\overline{\partial}^2\alpha^*)(\partial^2
\alpha h^{-1} +2\partial^{1}\alpha\partial
h^{-1}+ \alpha \partial^2h^{-1})\alpha^*\\
&=&\alpha(\overline{\partial}^2h^{-1}\partial^2
h+2\overline{\partial}^2h^{-1}\partial h \partial
h^{-1}h+\overline{\partial}^2h^{-1}h\partial^2h^{-1}h+2\overline{\partial}h^{-1}\overline{\partial}\partial^2h\\
&&+4\overline{\partial}h^{-1}\partial\overline{\partial}h\partial
h^{-1}h+2\overline{\partial}h^{-1}\overline{\partial}h\partial^2h^{-1}h+h^{-1}\overline{\partial}^2\partial^2h\\
&&+2h^{-1}\overline{\partial}^2\partial h \partial h^{-1}h+h^{-1}\overline{\partial}^2h\partial^2h^{-1}h)h^{-1}\alpha^*\\
&=&\alpha(\overline{\partial}^2h^{-1}\partial^2
h+2\overline{\partial}^2h^{-1}\partial h \partial
h^{-1}h+2\overline{\partial}h^{-1}\overline{\partial}^2\partial
h+4\overline{\partial}h^{-1}\partial\overline{\partial}h \partial
h^{-1}h\\
&&+h^{-1}\overline{\partial}^2\partial^2h+2h^{-1}\overline{\partial}^2\partial
h \partial h^{-1} h) h^{-1}\alpha^*.
\end{array}\tag{2.9}\label{2.9}\end{align*}

Recall that
$$\begin{array}{lll}-(K_P)_z&=&-(\partial
K_P+[\theta_P,K_P])\\
&=&\overline{\partial} h^{-1}\partial^2
h+2\overline{\partial}h^{-1}\partial h\partial h^{-1}
h+2h^{-1}\partial\overline{\partial}h \partial h^{-1}
h+h^{-1}\partial^2\overline{\partial}h.
\end{array}$$
and
\begin{align*}\begin{array}{lll}-(K_P)_{z\overline{z}}=-\overline{\partial}((K_P)_{z})&=&\overline{\partial}(\overline{\partial}
h^{-1}\partial^2 h+2\overline{\partial}h^{-1}\partial h\partial
h^{-1} h+2h^{-1}\partial\overline{\partial}h \partial h^{-1}
h+h^{-1}\partial^2\overline{\partial}h)\\
&=&\overline{\partial}^2h^{-1}\partial^2 h+\overline{\partial}h^{-1}\overline{\partial}\partial^2 h+2\overline{\partial}^2h^{-1}\partial h\partial h^{-1}h+2\overline{\partial}h^{-1}\partial\overline{\partial}h\partial h^{-1}h\\
&&+2\overline{\partial}h^{-1}\partial h \partial \overline{\partial}h^{-1}h+2\overline{\partial}h^{-1}\partial h \partial h^{-1} \overline{\partial}h +2\overline{\partial}h^{-1}\partial\overline{\partial}h\partial h^{-1}h\\
&&+2\overline{\partial}h^{-1}\partial\overline{\partial}^2h\partial
h^{-1}h+2h^{-1}\partial\overline{\partial}h\partial\overline{\partial}h^{-1}h
+2h^{-1}\partial\overline{\partial}h \partial
h^{-1}\overline{\partial}h\\
&&+\overline{\partial}h^{-1}\partial^2\overline{\partial}h+h^{-1}\partial^2\overline{\partial}^2
h.
\end{array}\tag{2.10}\label{2.10}\end{align*}

Note that \begin{align*}\begin{array}{llll}\
-2K^2_{P}&=&2(\overline{\partial}(h^{-1}\partial
h))(\overline{\partial}(\partial
h^{-1}h))\\&=&2(\overline{\partial}h^{-1}\partial h+h^{-1}\partial
\overline{\partial}h)(\partial \overline{\partial}h^{-1}h+\partial
h^{-1}\overline{\partial}h)\\
&=& 2(\overline{\partial}h^{-1}\partial
h
\partial \overline{\partial}h^{-1} h+ h^{-1}\partial
\overline{\partial}h
\partial\overline{\partial}h^{-1}h+\overline{\partial}h^{-1}\partial
h\partial h^{-1}\overline{\partial}h\\&&+h^{-1}\partial
\overline{\partial}h\partial h^{-1} \overline{\partial}h)
\end{array} \tag{2.11}\label{2.11}
\end{align*}

By formulaes \ref{2.9}, \ref{2.10} and \ref{2.11},  we can obtain
$$\begin{array}{lll}
\overline{\partial}^2P\partial^2
P+2\alpha(-K^2_{P})h^{-1}\alpha^*&=&
\overline{\partial}^2P\partial^2
P-2\alpha(-K_{P})h^{-1}\alpha^*\alpha(-K_{P})h^{-1}\alpha^*\\
&=&\overline{\partial}^2P\partial^2
P-2(\overline{\partial} P \partial P)^2\\
&=&-\alpha((K_P)_{z,\overline{z}})h^{-1}\alpha^*
\end{array}$$

Then it follows that
$$
\overline{\partial}^2 P
\partial^2 P-2(\overline{\partial} P \partial
P)^2=\alpha(-(K_P)_{z\overline{z}})h^{-1}\alpha^*.$$

 This finishes the proof of Claim 2.

By a direct computation, we can see that
 $${\mathscr K}_{0,0}(P)=\overline{\partial}P\partial P,
 {\mathscr K}_{1,0}(P)=\overline{\partial}^2P\partial P,
 {\mathscr K}_{0,1}(P)=\overline{\partial}P\partial^2 P, {\mathscr K}_{1,1}(P)=\overline{\partial}^2P\partial^2
 P-2(\overline{\partial}P\partial P)^2.$$
So conclusion of the lemma holds for $n=1$, by induction proof, we
assume that conclusion holds for $n\leq k$ and we will prove it also
holds for $n=k+1$ in the following.

Recall that
$$P(\lambda)=\alpha(\lambda)\cdot(\alpha^{*}(\lambda)\cdot\alpha(\lambda))^{-1}\cdot\alpha^{*}(\lambda)=\alpha(\lambda)\cdot
h^{-1}(\lambda)\cdot\alpha^{*}(\lambda), \forall \lambda\in
\Omega.$$

Set $$\partial P=F_1+F_2, F_1=\partial \alpha h^{-1}\alpha^*,
F_2=\alpha \partial h^{-1}\alpha^*;$$
$$\overline{\partial} P=G_1+G_2, G_1= \alpha  \overline{\partial} h^{-1}\alpha^*,
G_2=\alpha h^{-1} \overline{\partial}\alpha^*.$$

Now suppose that $i=k,$ or $j=k$ and
$${\mathscr K}_{i,j}(P)(\lambda)=\alpha(\lambda)(-(K_{P})_{z^i,{\overline
 z}^j})h^{-1}\alpha^*(\lambda), \forall \lambda\in \Omega.$$

 Then we have
 \begin{align*}\begin{array}{lll}
 \partial({\mathscr K}_{i,j}(P))&=&\partial (\alpha (-(K_{P})_{z^i,{\overline
 z}^j})h^{-1}\alpha^*)\\
 &=&\partial \alpha (-(K_{P})_{z^i,{\overline
 z}^j})h^{-1}\alpha^*+\alpha (-\partial (K_{P})_{z^i,{\overline
 z}^j})h^{-1}\alpha^*+\alpha (-(K_{P})_{z^i,{\overline
 z}^j})\partial h^{-1}\alpha^*.
 \end{array}\tag{2.12}\label{2.12}\end{align*}

 Note that
 \begin{align*}\begin{array}{llll}
 F_1({\mathscr K}_{i,j}(P))&=&\partial \alpha h^{-1}\alpha^*  (\alpha (-(K_{P})_{z^i,{\overline
 z}^j})h^{-1}\alpha^*)\\
 &=&\partial \alpha  (-(K_{P})_{z^i,{\overline
 z}^j})h^{-1}\alpha^*
 \end{array}\tag{2.13}\label{2.13}\end{align*}
 and
 \begin{align*}\begin{array}{llll}
 ({\mathscr K}_{i,j}(P))F_2&=&  (\alpha (-(K_{P})_{z^i,{\overline
 z}^j})h^{-1}\alpha^*) \alpha \partial h^{-1}\alpha^*\\
 &=& \alpha  (-(K_{P})_{z^i,{\overline
 z}^j})\partial h^{-1}\alpha^*\\
 &=& \alpha  (-(K_{P})_{z^i,{\overline
 z}^j})(-\theta_{P}) h^{-1}\alpha^*
 \end{array}\tag{2.14}\label{2.14}\end{align*}
and
\begin{align*}\begin{array}{llll}
 ({\mathscr K}_{i,j}(P))F_1&=&  (\alpha (-(K_{P})_{z^i,{\overline
 z}^j})h^{-1}\alpha^*) \partial \alpha h^{-1}\alpha^*\\
 &=& \alpha  (-(K_{P})_{z^i,{\overline
 z}^j})h^{-1}\partial h h^{-1}\alpha^*\\
  &=& \alpha  (-(K_{P})_{z^i,{\overline
 z}^j})\theta_{P} h^{-1}\alpha^*
 \end{array}\tag{2.15}\label{2.15}\end{align*}
 and
  \begin{align*}\begin{array}{llll}
 F_2({\mathscr K}_{i,j}(P))&=&  \alpha \partial h^{-1}\alpha^*(\alpha (-(K_{P})_{z^i,{\overline
 z}^j})h^{-1}\alpha^*)\\
 &=& \alpha \partial h^{-1}h (-(K_{P})_{z^i,{\overline
 z}^j})h^{-1}\alpha^*\\
 &=& \alpha(-\theta_{P})(-(K_{P})_{z^i,{\overline
 z}^j})h^{-1}\alpha^*.\\
 \end{array}\tag{2.16}\label{2.16}\end{align*}

By formulaes \ref{2.12}, \ref{2.13} and \ref{2.14}, we have
\begin{align*}\partial({\mathscr K}_{i,j}(P))=\alpha (-\partial (K_{P})_{z^i,{\overline
 z}^j})h^{-1}\alpha^*+ F_1(F_{i,j}(P))+ ({\mathscr K}_{i,j}(P))F_2.\tag{2.17}\label{2.17}\end{align*}

By formulaes \ref{2.15} and \ref{2.16}, we have
\begin{align*} \alpha[\theta_{P}, -(K_{P})_{z^i,\overline{z}^j}]h^{-1}\alpha^*=-F_2({\mathscr K}_{i,j}(P))-({\mathscr K}_{i,j}(P))F_1. \tag{2.18}\label{2.18}\end{align*}
Thus, it follows that
$$\begin{array}{llll}\alpha
(-(K_P)_{z^{i+1},\overline{z}^{j}})h^{-1}\alpha^{*}&=&\alpha((-\partial
(K_P)_{z^{i},\overline{z}^{j}})+[\theta,
-(K_P)_{z^{i},\overline{z}^{j}}])h^{-1}\alpha^*\\
&=&\partial({\mathscr K}_{i,j}(P))- F_1({\mathscr K}_{i,j}(P))-
({\mathscr K}_{i,j}(P))F_2-F_2({\mathscr K}_{i,j}(P))-({\mathscr K}_{i,j}(P))F_1\\
&=&\partial({\mathscr K}_{i,j}(P))-\partial P({\mathscr
K}_{i,j}(P))-({\mathscr K}_{i,j}(P))\partial P.
\end{array}
$$

In order to satisfy the conclusion of Lemma,  we need ${\mathscr
K}_{i,j}(P)$ have the following induction formulae :
\begin{align*} {\mathscr K}_{i+1,j}(P)=\partial({\mathscr K}_{i,j}(P))-\partial P({\mathscr K}_{i,j}(P))-({\mathscr K}_{i,j}(P))\partial
P, i,j=0,1,\cdots.\tag{2.19}\label{2.19}\end{align*}

On the other hand, we have
 \begin{align*}\begin{array}{lll}
\overline{ \partial}({\mathscr K}_{i,j}(P))&=&\overline{\partial}
(\alpha (-(K_{P})_{z^i,{\overline
 z}^j})h^{-1}\alpha^*)\\
 &=& \alpha (-\overline{\partial}(K_{P})_{z^i,{\overline
 z}^j})h^{-1}\alpha^*+\alpha ( -(K_{P})_{z^i,{\overline
 z}^j})\overline{\partial}h^{-1}\alpha^*+\alpha (-(K_{P})_{z^i,{\overline
 z}^j})h^{-1}\overline{\partial}\alpha^*\\
  &=& \alpha (-(K_{P})_{z^i,{\overline
 z}^{j+1}})h^{-1}\alpha^*+\alpha ( -(K_{P})_{z^i,{\overline
 z}^j})\overline{\partial}h^{-1}\alpha^*+\alpha (-(K_{P})_{z^i,{\overline
 z}^j})h^{-1}\overline{\partial}\alpha^*\\
 \end{array}\tag{2.20}\label{2.20}\end{align*}

 Note that
 \begin{align*}\begin{array}{llll}
G_1({\mathscr K}_{i,j}(P))&=& \alpha
\overline{\partial}h^{-1}\alpha^*  (\alpha (-(K_{P})_{z^i,{\overline
 z}^j})h^{-1}\alpha^*)\\
 &=& \alpha \overline{\partial}h^{-1} h (-(K_{P})_{z^i,{\overline
 z}^j})h^{-1}\alpha^*
 \end{array}\tag{2.21}\label{2.21}\end{align*}
 and
 \begin{align*}\begin{array}{llll}
 ({\mathscr K}_{i,j}(P))G_2&=&  (\alpha (-(K_{P})_{z^i,{\overline
 z}^j})h^{-1}\alpha^*) \alpha h^{-1} \overline{\partial}\alpha^*\\
 &=& \alpha  (-(K_{P})_{z^i,{\overline
 z}^j})h^{-1}\alpha^*\alpha h^{-1} \overline{\partial}\alpha^*\\
 &=& \alpha  (-(K_{P})_{z^i,{\overline
 z}^j}) h^{-1}\overline{\partial}\alpha^*\\
 \end{array}\tag{2.22}\label{2.22}\end{align*}
and
 \begin{align*}\begin{array}{llll}
 ({\mathscr K}_{i,j}(P))G_1&=&  (\alpha (-(K_{P})_{z^i,{\overline
 z}^j})h^{-1}\alpha^*)  \alpha \overline{\partial}h^{-1}\alpha^* \\
&=&  \alpha (-(K_{P})_{z^i,{\overline
 z}^j})h^{-1}h \overline{\partial}h^{-1}\alpha^* \\
  &=& \alpha  (-(K_{P})_{z^i,{\overline
 z}^j})\overline{\partial}h^{-1}\alpha^*
 \end{array}\tag{2.23}\label{2.23}\end{align*}
 and
  \begin{align*}\begin{array}{llll}
G_2({\mathscr K}_{i,j}(P))&=&  \alpha h^{-1}
\overline{\partial}\alpha^*(\alpha (-(K_{P})_{z^i,{\overline
 z}^j})h^{-1}\alpha^*)\\
 &=& \alpha h^{-1} \overline{\partial}h
(-(K_{P})_{z^i,{\overline
 z}^j})h^{-1}\alpha^*\\
 \end{array}\tag{2.24}\label{2.24}\end{align*}

By formulaes \ref{2.21} and \ref{2.22}, we have
$$\begin{array}{lll}
\overline{ \partial}({\mathscr K}_{i,j}(P))&=&  \alpha
(-(K_{P})_{z^i,{\overline
 z}^{j+1}})h^{-1}\alpha^*+({\mathscr K}_{i,j}(P))G_1+({\mathscr K}_{i,j}(P))G_2\\
\end{array}$$

By formulaes \ref{2.23} and \ref{2.24}, we have that
$$G_1({\mathscr K}_{i,j}(P))+G_2({\mathscr K}_{i,j}(P))=0.$$

Then we have
$$\begin{array}{lll}
\overline{ \partial}({\mathscr K}_{i,j}(P))&=&  \alpha
(-(K_{P})_{z^i,{\overline
 z}^{j+1}})h^{-1}\alpha^*+({\mathscr K}_{i,j}(P))G_1+({\mathscr K}_{i,j}(P))G_2\\
&=&  \alpha (-(K_{P})_{z^i,{\overline
 z}^{j+1}})h^{-1}\alpha^*+({\mathscr K}_{i,j}(P))G_1+({\mathscr K}_{i,j}(P))G_2+G_1({\mathscr K}_{i,j}(P))+G_2({\mathscr K}_{i,j}(P))\\
&=&  \alpha (-(K_{P})_{z^i,{\overline
 z}^{j+1}})h^{-1}\alpha^*+\overline{\partial}P({\mathscr K}_{i,j}(P))+({\mathscr K}_{i,j}(P))\overline{\partial}P\\
\end{array}$$

Then we also need  the following induction formulae :
\begin{align*} {\mathscr K}_{i,j+1}(P)=\overline{\partial}({\mathscr
K}_{i,j}(P))-\overline{\partial} P({\mathscr K}_{i,j}(P))-({\mathscr
K}_{i,j}(P))\overline{\partial} P, i,j=0,1,\cdots.
\tag{2.25}\label{2.25}
\end{align*}

To finish the proof, we only need to prove the following induction
formulaes \ref{2.19} and \ref{2.25} i.e.
$${\mathscr K}_{i+1,j}(P)=\partial({\mathscr K}_{i,j}(P))-\partial
P({\mathscr K}_{i,j}(P))-({\mathscr K}_{i,j}(P))\partial P,$$ $$
{\mathscr K}_{i,j+1}(P)=\overline{\partial}({\mathscr
K}_{i,j}(P))-\overline{\partial} P({\mathscr K}_{i,j}(P))-({\mathscr
K}_{i,j}(P))\overline{\partial} P, i,j=0,1,\cdots,$$

In this case, ${\mathscr K}_{i,j}(P)$, $i,j=0,1,\cdots$ will satisfy
the conclusion
$${\mathscr
K}_{i,j}(P)(\lambda)=\alpha(\lambda)(-K_{P,z^i,{\overline
 z}^j})h_1^{-1}\alpha^*(\lambda), \forall \lambda\in \Omega. $$

In order to prove this, we need the following  observations.

{\bf Claim 3}\,\,  Each ${\mathscr K}_{i,j}(P)$ for arbitrary $i,j$,
may be expressed by as a sum of monomials of the form
$$ (\overline{\partial }^{i_1}P\partial^{j_1}P)^{l_1}(\overline{\partial }^{i_2}P\partial^{j_2}P)^{l_2}\cdots(\overline{\partial
}^{i_t}P\partial^{j_t}P)^{l_t}.$$

By Claim 2, we already know that Claim 4 holds for the case of
$i,j\leq 2$. By the induction proof, we assume that the conclusion
holds for the case of $i,j\leq k$. Then we only need to prove the
conclusion also holds for the case of $i,j\leq k+1$.

With loss of generality, when $i=k+1$ or $j=k+1$,  we assume that
$${\mathscr K}_{i,j}(P)= (\overline{\partial }^{i_1}P\partial^{j_1}P)^{l_1}(\overline{\partial }^{i_2}P\partial^{j_2}P)^{l_2}\cdots(\overline{\partial
}^{i_t}P\partial^{j_t}P)^{l_t}.$$

Since ${\mathscr K}_{i,j+1}(P)$ is defined as
$\overline{\partial}({\mathscr K}_{i,j}(P))P.$ So we only need to
prove the conclusion will hold for $\overline{\partial}({\mathscr
K}_{i,j}(P))P.$ For the sake of simplicity of expression, we will
assume that
$${\mathscr K}_{i,j}(P)=(\overline{\partial }^{i_1}P\partial^{j_1}P)^{l_1}.$$
Then we have
 \begin{align*}\begin{array}{llll}
\overline{\partial}({\mathscr
K}_{i,j}(P))P&=&(\overline{\partial}^{i_1+1}P\partial^{j_1}P+\overline{\partial}^{i_1}P\overline{\partial}\partial^{j_1}P)
)(\overline{\partial
}^{i_1}P\partial^{j_1}P)^{l_1-1}+\cdots\\
&+&(\overline{\partial
}^{i_1}P\partial^{j_1}P)^{r-1}(\overline{\partial}^{i_1+1}P\partial^{j_1}P+\overline{\partial}^{i_1}P\overline{\partial}\partial^{j_1}P)
)(\overline{\partial
}^{i_1}P\partial^{j_1}P)^{l_1-r}+\cdots\\
&&+(\overline{\partial
}^{i_1}P\partial^{j_1}P)^{l_1-1}(\overline{\partial}^{i_1+1}P\partial^{j_1}P+\overline{\partial}^{i_1}P\overline{\partial}\partial^{j_1}P
).
\end{array}\tag{2.26}\label{2.26}
 \end{align*}

By \ref{1.2}, we have
$$\overline{\partial}\partial^{j_1}P=\partial^{j_1}P\overline{\partial}P-\overline{\partial}P\partial^{j_1}P-\sum\limits^{j_1-1}_{k=1}C^k_{j_1}(\partial^{j_1-k}P\overline{\partial}P\partial^kP)$$
And if $1<r<l$, then we have

$$\begin{array}{llll}
&&(\overline{\partial
}^{i_1}P\partial^{j_1}P)^{r-1}\partial(\overline{\partial}^{i_1}\partial^{j_1}P)(\overline{\partial
}^{i_1}P\partial^{j_1}P)^{l-r}\\&=& (\overline{\partial
}^{i_1}P\partial^{j_1}P)^{r-1}(\overline{\partial}^{i_1+1}P\partial^{j_1}P+\overline{\partial}^{i_1}P\overline{\partial}\partial^{j_1}P)
)(\overline{\partial }^{i_1}P\partial^{j_1}P)^{l-r}\\
&=& (\overline{\partial
}^{i_1}P\partial^{j_1}P)^{r-1}(\overline{\partial}^{i_1+1}P\partial^{j_1}P)(\overline{\partial
}^{i_1}P\partial^{j_1}P)^{l-r}\\
&&+(\overline{\partial
}^{i_1}P\partial^{j_1}P)^{r-1})(\overline{\partial}^{i_1}P\overline{\partial}\partial^{j_1}P)
(\overline{\partial
}^{i_1}P\partial^{j_1}P)^{l-r} \\
&=& (\overline{\partial
}^{i_1}P\partial^{j_1}P)^{r-1}(\overline{\partial}^{i_1+1}P\partial^{j_1}P)(\overline{\partial
}^{i_1}P\partial^{j_1}P)^{l-r}\\
&&+(\overline{\partial
}^{i_1}P\partial^{j_1}P)^{r-1}(\overline{\partial}^{i_1}P(\partial^{j_1}P\overline{\partial}P-\overline{\partial}P\partial^{j_1}P-\sum\limits^{j_1-1}_{k=1}C^k_{j_1}(\partial^{j_1-k}P\overline{\partial}P\partial^kP)))
(\overline{\partial
}^{i_1}P\partial^{j_1}P)^{l-r} \\

\end{array}$$

Since $\overline{\partial}^{i_1}P=P\overline{\partial}^{i_1}P$ and
$\overline{\partial}PP=0$, we have
$$\overline{\partial}^{i_1}P\partial^{j_1}P(\overline{\partial}^{i_1}P\partial^{j_1}P\overline{\partial}P)
(\overline{\partial}^{i_1}P\partial^{j_1}P)^{l-r}=\overline{\partial}^{i_1}P\partial^{j_1}P(\overline{\partial}^{i_1}P\partial^{j_1}P\overline{\partial}P\overline{\partial}P)
(P\overline{\partial}^{i_1}P\partial^{j_1}P)^{l-r}=0.$$ Similarly,
by the fact $\overline{\partial}^{i_1}\overline{\partial}P=0$, it
follows that
$$\overline{\partial}^{i_1}P\partial^{j_1}P(\overline{\partial}^{i_1}P\overline{\partial}P\partial^{j_1}P)
(\overline{\partial}^{i_1}P\partial^{j_1}P)^{l-r}=0.$$
 That means
  \begin{align*}\begin{array}{llll}
&&(\overline{\partial
}^{i_1}P\partial^{j_1}P)^{r-1}\partial(\overline{\partial}^{i_1}\partial^{j_1}P)(\overline{\partial
}^{i_1}P\partial^{j_1}P)^{l-r}\\&=& (\overline{\partial
}^{i_1}P\partial^{j_1}P)^{r-1}(\overline{\partial}^{i_1+1}P\partial^{j_1}P)(\overline{\partial
}^{i_1}P\partial^{j_1}P)^{l-r}\\
&&+(\overline{\partial
}^{i_1}P\partial^{j_1}P)^{r-1}(\overline{\partial}^{i_1}P(\partial^{j_1}P\overline{\partial}P-\overline{\partial}P\partial^{j_1}P-\sum\limits^{j_1-1}_{k=1}C^k_{j_1}(\partial^{j_1-k}P\overline{\partial}P\partial^kP)))
(\overline{\partial
}^{i_1}P\partial^{j_1}P)^{l-r} \\
&=& (\overline{\partial
}^{i_1}P\partial^{j_1}P)^{r-1}(\overline{\partial}^{i_1+1}P\partial^{j_1}P)(\overline{\partial
}^{i_1}P\partial^{j_1}P)^{l-r}\\
&&-\sum\limits^{j_1-1}_{k=1}C^k_{j_1}(\overline{\partial}^{i_1}P\partial^{j_1}P)^{r-1}(\overline{\partial}^{i_1}P\partial^{j_1-k}P\overline{\partial}P\partial^kP)
(\overline{\partial
}^{i_1}P\partial^{j_1}P)^{l-r} \\
\end{array}\tag{2.27}\label{2.27} \end{align*}

By the formulaes \ref{2.26} and \ref{2.27}, we can see that Claim 3
also holds for ${\mathscr K}_{i,j+1}(P)$.

 By a similar proof, we also can see that Claim 3 holds for
${\mathscr K}_{i+1,j}(P)$. Thus, we finish the proof of Claim 3.

Now we can prove the formulaes \ref{2.19} and \ref{2.25} as the
ending of the proof of the lemma.

By the Claim 3 and formulaes \ref{1.3} and \ref{1.4}, we will
obtain
$$\begin{array}{llll}{\mathscr K}_{i,j}(P)&=&(\overline{\partial }^{i_1}P\partial^{j_1}P)^{l_1}(\overline{\partial }^{i_2}P\partial^{j_2}P)^{l_2}\cdots(\overline{\partial
}^{i_t}P\partial^{j_t}P)^{l_t}\\
&=&(\overline{\partial
}^{i_1}P\partial^{j_1}P)^{l_1}(\overline{\partial
}^{i_2}P\partial^{j_2}P)^{l_2}\cdots(\overline{\partial
}^{i_t}P\partial^{j_t}P)^{l_t}P\\
&=&{\mathscr K}_{i,j}(P)P.\end{array}$$

By \ref{1.4}, it follows that $$\overline{\partial}P{\mathscr
K}_{i,j}(P)=0.$$
 Then we have
$$\begin{array}{llll}\overline{\partial}({\mathscr K}_{i,j}(P))-({\mathscr
K}_{i,j}(P))\overline{\partial}P-\overline{\partial}P({\mathscr
K}_{i,j}(P))&=&\overline{\partial}({\mathscr
K}_{i,j}(P)P)-({\mathscr
K}_{i,j}(P))\overline{\partial}P\\
&=&\overline{\partial}({\mathscr
K}_{i,j}(P))P\\
&=&{\mathscr K}_{i,j+1}(P)\\
\end{array}
$$
Thus, the formulae \ref{2.19} holds. Similarly, we also can prove
\ref{2.25}.

\end{proof}

\begin{rem} From the proof of Lemma \ref{main}, we can
 see that the curvature formulae ${\mathscr K}_{i,j}(P)$ (See Definition 2.1 ) does not depend on the
 chose of $P$.
 \end{rem}

\begin{lem}\label{tran} Let $P,Q\in {\mathcal P}_n(\Omega, {\mathcal
 U})\cap {\mathcal A}_n(\Omega,{\mathcal U})$. And there exist  $\alpha ,\beta: \Omega\rightarrow 
{\mathscr L}(\mathbb{C}^n, l^2(\mathbb{N},B))$  such that
$$P(\lambda)=\alpha(\lambda)\cdot(\alpha^{*}(\lambda)\cdot\alpha(\lambda))^{-1}\cdot\alpha^{*}(\lambda),
Q(\lambda)=\beta(\lambda)\cdot(\beta^{*}(\lambda)\cdot\beta(\lambda))^{-1}\cdot\beta^{*}(\lambda),
\forall \lambda \in \Omega.$$ Let ${\mathscr K}_{i,j}(P), {\mathscr
K}_{i,j}(Q)$ $ i,j=0,1,\cdots,n$ be differential $\mathcal U$-valued
functions in ${\mathcal
 U}$ constructed in Lemma \ref{main} according to $P$ and $Q$ respectively.

 Let $k\geq 1$ be an integer.  Then  there exists a unitary $v$
such that $$v\overline{\partial}^iP(\lambda)\partial^iP(\lambda)v^*=
\overline{\partial}^jQ(\lambda)\partial^iQ(\lambda),\forall i,j\leq
k$$ if and only if for any $\lambda\in \Omega$,
$$v{\mathscr K}_{i,j}(P)(\lambda)v^*=
{\mathscr K}_{i,j}(Q)(\lambda),\forall i,j\leq k.$$
 \end{lem}

\begin{proof} By Claim 4 of Lemma \ref{main}, each ${\mathscr K}_{i,j}(P)$  may be expressed by as a sum
of monomials of the form
$$ (\overline{\partial }^{i_1}P\partial^{j_1}P)^{l_1}(\overline{\partial }^{i_2}P\partial^{j_2}P)^{l_2}\cdots(\overline{\partial
}^{i_t}P\partial^{j_t}P)^{l_t}.$$

Firstly, we have the following claim:

{\bf Claim 1}\,\,In the expression formulae of ${\mathscr
K}_{i,j}(P)$,
 $\overline{\partial}^iP\partial^j P$ appears only once.

 In fact, when $0\leq i,j\leq 1$, we have
 $${\mathscr K}_{0,0}(P)=\overline{\partial}P\partial P,
 {\mathscr K}_{1,0}(P)=\overline{\partial}^2P\partial P,
 {\mathscr K}_{0,1}(P)=\overline{\partial}P\partial^2 P, {\mathscr K}_{1,1}(P)=\overline{\partial}^2P\partial^2
 P-2(\overline{\partial}P\partial P)^2.$$
If we  assume that the Claim 1 holds for $i,j\leq k$, $k\geq 1,$
then we only need to prove that Claim 1 will also hold for
$i,j=k+1.$

Since $$\begin{array}{lllll} {\mathscr
K}_{i+1,j}(P)&=&\partial({\mathscr K}_{i,j}(P))-\partial
P({\mathscr K}_{i,j}(P))-({\mathscr K}_{i,j}(P))\partial P\\
&=&\partial(\overline{\partial}^i\partial^j P+{\mathscr
K}_{i,j}(P)-\overline{\partial}^i\partial^j P)-\partial
P({\mathscr K}_{i,j}(P))-({\mathscr K}_{i,j}(P))\partial P\\
&=&\overline{\partial}^{i+1}P\partial^j(P)+\overline{\partial}^iP\overline{\partial}\partial^{j+1}P+\partial({\mathscr
K}_{i,j}(P)-\overline{\partial}^i\partial^j P)-\partial
P({\mathscr K}_{i,j}(P))-({\mathscr K}_{i,j}(P))\partial P\\
\end{array}
$$

If $\overline{\partial}^kP\partial^l P$ appears in the expression
formulae of $\partial({\mathscr
K}_{i,j}(P)-\overline{\partial}^i\partial^j P)-\partial P({\mathscr
K}_{i,j}(P))-({\mathscr K}_{i,j}(P))\partial P$, then $k<i$. So we
can see that $\overline{\partial}^{i+1}P\partial^j(P)$ appears only
once in the expression formulae of ${\mathscr K}_{i+1,j}(P)$. Then
we finish the proof of Claim 1.

{\bf Claim 2}\,\, Let $v$ be a unitary of $\mathcal U$. Then
$$v{\mathscr K}_{i,j}(P)v^*={\mathscr K}_{i,j}(Q), \forall i,j \leq k\Rightarrow
v\overline{\partial}^iP\partial^jPv^*=\overline{\partial}^iQ\partial^jQ,
\forall i,j\leq k.$$

In fact, when $0\leq i,j\leq 1,$ recall that
$${\mathscr K}_{0,0}(P)=\overline{\partial}P\partial P,
 {\mathscr K}_{1,0}(P)=\overline{\partial}^2P\partial P,
 {\mathscr K}_{0,1}(P)=\overline{\partial}P\partial^2 P, {\mathscr K}_{1,1}(P)=\overline{\partial}^2P\partial^2
 P-2(\overline{\partial}P\partial P)^2.$$

 If there exists unitary $v$ such that
 $$v{\mathscr K}_{i,j}(P)v^*={\mathscr K}_{i,j}(P), i,j\leq 2,$$
 then $$v(\overline{\partial}P\partial
 P)v^*=v({\mathscr K}_{0,0}(P))v^*={\mathscr K}_{0,0}(Q)=\overline{\partial}Q\partial
 Q,$$
$$ v(\overline{\partial}^2 P\partial P)v^*=v({\mathscr K}_{2,1}(P))v^*={\mathscr K}_{1,0}(Q)=\overline{\partial}^2 Q\partial Q,$$
$$v(\overline{\partial} P\partial^2 P)v^*=v({\mathscr K}_{1,2}(P))v^*={\mathscr K}_{0,1}(Q)=\overline{\partial} Q\partial^2 Q, $$
and
$$\begin{array}{llll}
v({\mathscr
K}_{1,1}(P))v^*&=&v(\overline{\partial}^2P\partial^2P)v^*-2v(\overline{\partial}P\partial
P)^2v^*\\
&=&v(\overline{\partial}^2P\partial^2P)v^*-2v(\overline{\partial}P\partial
P)v^*v(\overline{\partial}P\partial
P)v^*\\
&=&\overline{\partial}^2Q\partial^2Q-2(\overline{\partial}Q\partial
Q)^2.
\end{array}
$$
It follows that $$
v(\overline{\partial}^2P\partial^2P)v^*=\overline{\partial}^2Q\partial^2Q.$$
If we assume the Claim 2 holds for the case of ``$k\leq l$'', then
we only need to prove the conclusion holds for the case of $k=l+1.$

Note that ${\mathscr K}_{i,j}(P)$ may be expressed by  as a sum of
monomials of the form
$$ (\overline{\partial }^{i_1}P\partial^{j_1}P)^{l_1}(\overline{\partial }^{i_2}P\partial^{j_2}P)^{l_2}\cdots(\overline{\partial
}^{i_t}P\partial^{j_t}P)^{l_t}.$$ And
$\overline{\partial}^iP\partial^jP$ appears only once in the
expression formulae of ${\mathscr K}_{i,j}(P)$. Let ${\mathscr
K}_{i,j}(P)\sim {\mathscr K}_{i,j}(Q)$ i.e. there exists unitary $v$
such that
$$v{\mathscr K}_{i,j}(P)v^*={\mathscr K}_{i,j}(Q), i,j \leq l.$$

By induction proof, we have that
 \begin{align*}v\overline{\partial}^iP\partial^j
Pv^*=\overline{\partial}^iQ\partial^j Q, i,j\leq
l.\tag{2.29}\label{2.29} \end{align*}

And if $$v{\mathscr K}_{l+1,l}(P)v^*={\mathscr K}_{l+1,l}(Q),$$ then
we have
$$\begin{array}{llll}
v({\mathscr
K}_{l+1,l}(P))v^*&=&v(\overline{\partial}^{i+1}P\partial^j
P)v^*+v({\mathscr K}_{l+1,l}(P)-\overline{\partial}^{i+1}P\partial^j P)v^*\\
&=&{\mathscr K}_{l+1,l}(Q)\\
&=&\overline{\partial}^{i+1}Q\partial^j Q+{\mathscr
K}_{l+1,l}(Q)-\overline{\partial}^{i+1}Q\partial^j Q
\end{array}
$$
 Since ${\mathscr K}_{l+1,l}(P)-\overline{\partial}^{l+1}P\partial^l P$ may be expressed by  as a sum of monomials of
the form
$$ (\overline{\partial }^{\widetilde{i}_1}P\partial^{\widetilde{j}_1}P)^{l_1}(\overline{\partial }^{\widetilde{i}_2}P\partial^{\widetilde{j}_2}P)^{l_2}\cdots(\overline{\partial
}^{\widetilde{i}_{\widetilde{t}}}P\partial^{\widetilde{j}_{\widetilde{t}}}P)^{l_{\widetilde{t}}}.$$
 Since $\overline{\partial}^{l+1}P\partial^l P$ appears only once in the expression formulae of ${\mathscr K}_{l+1,l}(P)$,
 we have $$\widetilde{i}_{r}, \widetilde{j}_{r}\leq l,
r\leq\widetilde{t}.$$  By formulae \ref{2.29}, we have
$$v({\mathscr K}_{l+1,l}(P)-\overline{\partial}^{l+1}P\partial^l P)v^*={\mathscr K}_{l+1,l}(Q)-\overline{\partial}^{l+1}Q\partial^l
Q.$$

So we have $$v(\overline{\partial}^{l+1}P\partial^l
P)v^*=\overline{\partial}^{l+1}Q\partial^l Q.$$

Similarly, we can prove that
$$v(\overline{\partial}^{l}P\partial^{l+1}
P)v^*=\overline{\partial}^{l}Q\partial^{l+1} Q,$$ and
$$v(\overline{\partial}^{l+1}P\partial^{l+1}
P)v^*=\overline{\partial}^{l+1}Q\partial^{l+1} Q.$$ Then we finish
the proof of Claim 2.

{\bf Claim 3}\,\, Let $v$ be a unitary of $\mathcal U$. Then
$$v{\mathscr K}_{i,j}(P)v^*={\mathscr K}_{i,j}(Q), \forall i,j \leq k\Leftarrow
v\overline{\partial}^iP\partial^jPv^*=\overline{\partial}^iQ\partial^jQ,
\forall i,j\leq k.$$

Suppose that ${\mathscr K}_{i,j}(P)$ is expressed by  as a sum of
monomials of the form
$$ (\overline{\partial }^{i_1}P\partial^{j_1}P)^{l_1}(\overline{\partial }^{i_2}P\partial^{j_2}P)^{l_2}\cdots(\overline{\partial
}^{i_t}P\partial^{j_t}P)^{l_t},$$ and $i_r, j_r \leq k, r\leq t.$

Then we have
$$\overline{\partial}^{i_r}Q\partial^{j_r}Q=v\overline{\partial}^{i_r}P\partial^{j_r}Pv^*,
r\leq t,$$ and
$$ v(\overline{\partial }^{i_1}P\partial^{j_1}P)^{l_1}(\overline{\partial }^{i_2}P\partial^{j_2}P)^{l_2}\cdots(\overline{\partial
}^{i_t}P\partial^{j_t}P)^{l_t}v^*=(\overline{\partial
}^{i_1}Q\partial^{j_1}Q)^{l_1}(\overline{\partial
}^{i_2}Q\partial^{j_2}Q)^{l_2}\cdots(\overline{\partial
}^{i_t}Q\partial^{j_t}Q)^{l_t}.$$

Then we finish the proof of Lemma \ref{tran}.

\end{proof}

\begin{thm}\label{mainthm1}
Let $P,Q\in {\mathcal A}_n(\Omega,{\mathcal U})$ be two 
holomorphic curves. Then $P\sim_{u} Q$ if and only if ${\mathscr
K}_{i,j}(P)(\lambda)\sim_{u} {\mathscr K}_{i,j}(Q)(\lambda), \forall
\lambda\in \Omega, \,\,\mbox{and}\,\, i,j=0,1,\cdots,n-1.$
\end{thm}

\begin{proof}

By Theorem \ref{MSlemma}, we have  $P\sim_{u}Q$ if and only if for
each $\lambda\in \Omega$, there exists a unitary $v_{\lambda}\in
{\mathcal U}$ such that
$$v_{\lambda}\overline{\partial}^iP(\lambda)\partial^jP(\lambda)v_{\lambda}^*=
\overline{\partial}^iQ(\lambda)\partial^jQ(\lambda),\forall i,j\leq
k.$$ By lemma \ref{tran}, we have
$$v_{\lambda}\overline{\partial}^iP(\lambda)\partial^iP(\lambda)v_{\lambda}^*=
\overline{\partial}^jQ(\lambda)\partial^iQ(\lambda),\forall i,j\leq
k$$ if and only if for any $\lambda\in \Omega$,
$$v_{\lambda}{\mathscr K}_{i,j}(P)(\lambda)v_{\lambda}^*=
{\mathscr K}_{i,j}(Q)(\lambda),\forall i,j\leq k.$$  By Lemma
\ref{MSlemma}, we finish the proof of Theorem \ref{mainthm1}.
\end{proof}

Let $P,Q\in {\mathcal P}_n(\Omega, {\mathcal
 U})\cap {\mathcal A}_n(\Omega,{\mathcal U})$.
 And there exist holomorphic maps $\alpha, \beta: \Omega\rightarrow 
{\mathscr L}(\mathbb{C}^n, l^2(\mathbb{N},B))$ such that
$$P(\lambda)=\alpha(\lambda)\cdot(\alpha^{*}(\lambda)\cdot\alpha(\lambda))^{-1}\cdot\alpha^{*}(\lambda),
Q(\lambda)=\beta(\lambda)\cdot(\beta^{*}(\lambda)\cdot\beta(\lambda))^{-1}\cdot\beta^{*}(\lambda),
\forall \lambda \in \Omega.$$ {\bf Let  ${\mathscr K}_{i,j}(P),
{\mathscr K}_{i,j}(Q): \Omega\rightarrow {\mathcal U}$, $
i,j=0,1,\cdots,n$ be the curvatures and covariant derivatives of $P$
and $Q$ respectively. By Lemma \ref{main}, we have that
$${\mathscr
K}_{i,j}(P)=\alpha(-K_{P,z^i,{\overline
 z}^j})h_1^{-1}\alpha^*, {\mathscr K}_{i,j}(Q)=\beta(-K_{Q,z^i,{\overline
 z}^j})h_2^{-1}\beta^*, $$ where
 $h_1=\alpha^{*}\cdot\alpha, h_2=\beta^*\cdot \beta.$}

\begin{thm}\label{mainthm2}
Let $P,Q\in {\mathcal P}_n(\Omega, {\mathcal
 U})\cap {\mathcal A}_n(\Omega,{\mathcal U})$. Then the following statements are equivalent

$(1)$\,\,$P\sim_{u}Q$

$(2)$\,\, ${\mathscr K}_{i,j}(P)(\lambda)\sim_{u} {\mathscr
K}_{i,j}(Q)(\lambda), \forall \lambda\in \Omega, \,\,\mbox{and}\,\,
i,j=0,1,\cdots,n-1.$

$(3)$\,\,There exists $\lambda_0\in
 \Omega$ and invertible operator $X_{\lambda_0}$ such that
 \begin{enumerate}
 \item[(i)] $X_{\lambda_0}(\alpha_1(-K_{P,z^i,{\overline
 z}^j})h_1^{-1}\alpha_1^*)(\lambda_0)= (\alpha_2(-K_{Q,z^i,{\overline
 z}^j})h_2^{-1}\alpha_2^*(\lambda_0))X_{\lambda_0}, 0\leq i,j,$;
\item[(ii)] $ X_{\lambda_0} \partial^{i}P(\lambda_0)\overline{\partial}^jP(\lambda_0)= \partial^{i}Q(\lambda_0)\overline{\partial}^jQ(\lambda_0)X_{\lambda_0}, 0\leq
i,j;$
  \end{enumerate}

\end{thm}

\begin{proof} By Theorem \ref{mainthm1}, we only need to prove the
equivalence of (1) and (3).  If there exists an invertible operator
$X\in {\mathcal L}({\mathcal H})$ such that
$XP_1(\lambda)=P_2(\lambda)X$, then we have
$$X\partial^iP_1(\lambda)\overline{\partial}^iP_1(\lambda)X^{-1}=\partial^iP_2(\lambda)\overline{\partial}^jP_2(\lambda),
\forall \lambda\in \Omega, i,j\in Z^{\infty}_+.$$ By the similar
proof of Lemma \ref{tran}, we also obtain $X{\mathscr
K}_{i,j}(P_1)(\lambda)={\mathscr K}_{i,j}(P_2)(\lambda)X,\forall
i,j\in Z^{\infty}_+$. Thus, by the curvature formulae in Lemma
\ref{main}, the sufficient part follows.

On the other hand, for a fixed $\lambda_0\in \Omega$, there exists
an invertible operator $X_{\lambda_0}$ such that
$$X_{\lambda_0}(\alpha_1(-K_{P,z^i,{\overline
 z}^j})h_1^{-1}\alpha_1^*)(\lambda_0)= (\alpha_2(-K_{Q,z^i,{\overline
 z}^j})h_2^{-1}\alpha_2^*(\lambda_0))X_{\lambda_0}, 0\leq i,j,$$
that is equivalent to say that $X_{\lambda_0}{\mathscr
K}_{i,j}(P_1)(\lambda_0)={\mathscr
K}_{i,j}(P_2)(\lambda_0)X_{\lambda_0},\forall i,j\in Z^{\infty}_+$.
Then we have that
$$X_{\lambda_0}\partial^iP_1(\lambda_0)\overline{\partial}^jP_1(\lambda_0)=\partial^iP_2(\lambda_0)\overline{\partial}^jP_2(\lambda_0)X_{\lambda_0},
\forall i,j\in Z^{\infty}_+.$$ Suppose $P$ is an 
holomorphic curve. By recalling the formulae \ref{1.2} and
\ref{1.3}, we know for any $I, J\in Z^{\infty}_{+}$, and any
$\lambda\in\Omega$, $\partial^{I}\overline{\partial}^{J}P(\lambda)$
can be represented by a sum of monnomials of the form (the form does
not depend on $P$) $\pm
[\overline{\partial}^{I_1}P][\partial^{J_1}P]\cdots[\overline{\partial}^{I_k}P][\partial^{J_k}P]$
and $\pm
[\partial^{J_1}P][\overline{\partial}^{I_1}P]\cdots[\partial^{J_k}P][\overline{\partial}^{I_k}P]$.
Thus, we also have
that$$X_{\lambda_0}\partial^i\overline{\partial}^jP_1(\lambda_0)=\partial^i\overline{\partial}^jP_2(\lambda_0)X_{\lambda_0},
\forall i,j\in Z^{\infty}_+.$$

Note that there exists $\Omega_0 \subset \Omega$ which is a open
neighborhood of $\lambda_0$, such that
$P_i(\lambda)=\sum\limits_{i,j=0}^{\infty}\frac{\partial^i\overline{\partial}^jP_i(\lambda_0)}{i!j!}(\lambda-\lambda_0)^i(\overline{\lambda}-\overline{\lambda_0})^j,
\forall \lambda\in\Omega, i=1,2$.  Thus, this finishes the proof of
sufficient part.
\end{proof}

\begin{cor}

Let  $P_1, P_2$ be the  holomorphic curves defined in the
theorem above 2.9. If $P_1\sim_{u} P_2$, then $K_{P_1,z^i,{\overline
 z}^j}(\lambda)\sim_s K_{P_2,z^i,{\overline
 z}^j}(\lambda), \forall \lambda\in \Omega$.

\end{cor}

\begin{proof}
When $P_1\sim_{s} P_2$, i,e. there exists unitary operator $U$ such
that $UP_1(\lambda)U^{*}=P_2(\lambda)$ for any $\lambda\in \Omega$.
Furthermore, $U{\mathscr K}_{i,j}(P_1)(\lambda)={\mathscr
K}_{i,j}(P_2)(\lambda)U,\forall i,j\in Z^{\infty}_+$. And
$$U(\alpha_1(-K_{E_1,z^i,{\overline
 z}^j})h_1^{-1}\alpha_1^*)= (\alpha_2(-K_{E_2,z^i,{\overline
 z}^j})h_2^{-1}\alpha_2^*U, 0\leq i,j.$$
Then for any $\lambda\in \Omega$ and $i,j\in Z^{+}_{\infty}$, we
obtain
$$
U\alpha_1(\lambda)(-K_{E_1,z^i,{\overline
 z}^j}(\lambda))h_1^{-1}(\lambda)\alpha_1^*(\lambda)=\alpha_2(\lambda)(-K_{E_2,z^i,{\overline
 z}^j}(\lambda))h_2^{-1}(\lambda)\alpha_2^*(\lambda)U$$
 By multiplying $\alpha^*_2(\lambda)$ and $\alpha_1(\lambda)$ on the
 both sides of the formula above respectively,
$$ \alpha^*_2(\lambda)U\alpha_1(\lambda)(-K_{E_1,z^i,{\overline
 z}^j}(\lambda))h_1^{-1}(\lambda)\alpha_1^*(\lambda)\alpha_1(\lambda)=\alpha_2^*(\lambda)\alpha_2(\lambda)(-K_{E_2,z^i,{\overline
 z}^j}(\lambda))h_2^{-1}(\lambda)\alpha_2^*(\lambda)U\alpha_1(\lambda)$$
That means
$$ h^{-1}_2(\lambda)\alpha^*_2(\lambda)U\alpha_1(\lambda)K_{E_1,z^i,{\overline
 z}^j}(\lambda)=K_{E_2,z^i,{\overline
 z}^j}(\lambda)h_2^{-1}(\lambda)\alpha_2^*(\lambda)U\alpha_1(\lambda).$$
 Then we only need to prove
 $Y_{\lambda}:=h^{-1}_2(\lambda)\alpha^*_2(\lambda)U\alpha_1(\lambda)$ is an
 invertible matrix for any $\lambda\in \Omega$.  Now we set
 $Z_{\lambda}:=h^{-1}_1(\lambda)\alpha^*_1(\lambda)U^{*}\alpha_2(\lambda)$.
By $UP_1(\lambda)U^{*}=P_2(\lambda)$,
 then we obtain
 $$\begin{array}{llll}
 Y_{\lambda}Z_{\lambda}&=&h^{-1}_2(\lambda)\alpha^*_2(\lambda)U\alpha_1(\lambda)h^{-1}_1(\lambda)\alpha^*_1(\lambda)U^{*}\alpha_2(\lambda) \\
 &=&h^{-1}_2(\lambda)\alpha^*_2(\lambda)U\alpha_1(\lambda)h^{-1}_1(\lambda)\alpha^*_1(\lambda)U^{*}\alpha_2(\lambda) \\
 &=&h^{-1}_2(\lambda)\alpha^*_2(\lambda)UP_1(\lambda)U^{*}\alpha_2(\lambda) \\
 &=&h^{-1}_2(\lambda)\alpha^*_2(\lambda)P_2(\lambda)\alpha_2(\lambda) \\
 &=&h^{-1}_2(\lambda)\alpha^*_2(\lambda)(\alpha_2(\lambda)h^{-1}_2(\lambda)\alpha^*_2(\lambda))\alpha_2(\lambda) \\
 &=&h^{-1}_2(\lambda)h_2(\lambda)h^{-1}_2(\lambda)h_2(\lambda) \\
 &=&I_n
 \end{array}$$
 Similarly, we also can prove that $Z_{\lambda}Y_{\lambda}=I_n$.
 Note that $Y_{\lambda}K_{E_1,z^i,{\overline
 z}^j}(\lambda)=K_{E_2,z^i,{\overline
 z}^j}(\lambda)Y_{\lambda}, \forall \lambda\in \Omega,$ this finishes the proof.

\end{proof}

\subsection{ Extended holomorphic curve}

For any ${\mathcal U}$ a unital $C^*$-algebra, let ${\mathcal I}(
{\mathcal U})$ denote the set of all of the idempotents in
${\mathcal U}$.  ${\mathcal I}({\mathcal U})$  is called as the
extended Grassmann manifold and it's properties can be found in
\cite{MS3,S}. In the end of this section, we will generalized the
results on the holomorphic curves on Grassmann manifold of
$C^*$-algebras to the holomorphic curves on extended Grassmann
manifold case.

\begin{defn}
We say a real-analytic map ${\mathcal I}:\Omega\rightarrow {\mathcal
I}({\mathcal U})$ as an extended holomorphic curve if the following
statements hold $$ \overline{\partial} {\mathcal I}(\lambda)={\mathcal
I}(\lambda)\overline{\partial} {\mathcal I}(\lambda), \partial{\mathcal I}(\lambda)=\partial{\mathcal I}(\lambda){\mathcal
I}(\lambda),\overline{\partial} {\mathcal I}(\lambda){\mathcal
I}(\lambda)=0,{\mathcal I}(\lambda){\partial} {\mathcal
I}(\lambda)=0, \forall \lambda \in \Omega.$$
\end{defn}

\begin{defn}\label{Ide}

For ${\mathcal U}={\mathscr L}(l^2(\mathbb{N},B))$,  let ${\mathcal
I}_n(\Omega ,{\mathcal U})$ denotes the extended holomorphic curve
${\mathcal I}:\Omega\rightarrow {\mathcal U}$ which satisfies:
 ${\mathcal I}(\lambda)=\alpha(\lambda)\cdot(\beta^{*}(\lambda)\cdot\alpha(\lambda))^{-1}\cdot\beta^{*}(\lambda),
\forall \lambda \in \Omega,$ where
 $\alpha,\beta : \Omega\rightarrow 
{\mathscr L}(\mathbb{C}^n, l^2(\mathbb{N},B))$ as follows 
$$\alpha(\lambda)(w_1,w_2,\cdots,w_n)=\sum\limits_{i=1}^nw_i\alpha_i(\lambda), $$ 
$$\beta(\lambda)(w_1,w_2,\cdots,w_n)=\sum\limits_{i=1}^nw_i\beta_i(\lambda)$$
and  
 $\alpha_i, \beta_i:\Omega \rightarrow 
l^2(\mathbb{N},B), i=1,2\cdots,n$ are holomorphic functions.

\end{defn}

Similar to the proof of Lemma \ref{hlc}, we have the following proposition: 

 \begin{prop}
Let $B$ be a unital $C^*$-algebra and ${\mathcal U}={\mathscr
L}(l^2(\mathbb{N},B))$. Then the $C^{\infty}$ map ${\mathcal
I}:=\alpha\cdot(\beta^{*}\cdot\alpha)^{-1}\cdot\beta^{*}$ defined in
definition \ref{Ide} is an extended holomorphic curve.
\end{prop}

\begin{defn}\label{GECur}
 Let ${\mathcal I}\in {\mathcal
I}_n(\Omega ,{\mathcal U})$ be an extended holomorphic curve. We say
${\mathscr K}_{i,j}(\mathcal I), 0\leq i,j$ defined as the following
to be the extended curvature and covariant derivatives of curvature
of the extended holomorphic curve $\mathcal I$ on the extended
Grassmann manifold  ${\mathcal \mathcal I({\mathcal U})}$, where
 \begin{align*}  {\mathscr K}(\mathcal I):={\mathscr K}_{0,0}(\mathcal I)=\overline{\partial}\mathcal I{\partial \mathcal I},
  {\mathscr K}_{i+1,j}(\mathcal I)=\mathcal I(\partial ({\mathscr K}_{i,j}(\mathcal I))),
{\mathscr K}_{i,j+1}(\mathcal I)=(\overline{\partial}({\mathscr
K}_{i,j}(\mathcal I)))\mathcal I, \forall i,j\in \mathbb{Z}_{+}.
\end{align*}

\end{defn}

\begin{defn}\label{excurvature1}  Let ${\mathcal I}\in {\mathcal
I}_n(\Omega ,{\mathcal U})$  be an extended holomorphic curve on
extended Grassmann manifold. Set
 $$h(\lambda)=<\alpha(\lambda),\beta(\lambda)>=\beta^{*}(\lambda)\cdot\alpha(\lambda),$$
 then extended curvature function of $ {\mathcal I}$ is defined as
$$K_{\mathcal I}=-\frac{\partial}{\partial \overline{\lambda}}(h^{-1}\frac{\partial h}{\partial \lambda}), \mbox{for all}~ \lambda\in \Omega.$$
And the partial derivatives of extended curvature are defined as the
following:

 (1)\,\,$K_{{\mathcal I},\overline{\lambda}}=\frac{\partial}{\overline{\partial}\lambda}(K_{{\mathcal I}});$

 (2)\,\,$K_{{\mathcal I},\lambda}=\frac{\partial}{\partial \lambda}(K_{\mathcal I})+[h^{-1}\frac{\partial}{\partial \lambda}h,K_{{\mathcal I}}],~\mbox{for all}~ \lambda\in \Omega. $

By the definition above,  we can get the partial derivatives of
extended curvature: $K_{{\mathcal
I},\lambda^i\overline{\lambda}^j}$, $i,j\in \mathbb{N}\cup \{0\}$ by
using the inductive formulaes above.
\end{defn}

\begin{rem}\label{mainid} Let ${\mathcal I}\in {\mathcal
I}_n(\Omega ,{\mathcal U})$. By the similar proof of the projection
case, we also have the following statement:
$${\mathscr K}_{i,j}({\mathcal I})(\lambda)=\alpha(\lambda)(-K_{{\mathcal I},z^i,{\overline
 z}^j})h^{-1}\alpha^*(\lambda), \forall \lambda\in \Omega,$$
 where $h=\beta^*\cdot \alpha$.

\end{rem}

Similar to the proof of Theorem \ref{mainthm2}, we can obtain the
following similarity theorem of the extended holomorphic curve on
the extended Grassmann manifold.

\begin{thm}\label{mainthm3}
Let $\mathcal{I},\mathcal{J}\in {\mathcal I}_n(\Omega, {\mathcal
 U})$. Then the following statements are equivalent

$(1)$\,\,$\mathcal{I}\sim_{s}\mathcal{J}$

$(2)$\,\,There exists $\lambda_0\in
 \Omega$ and invertible operator $X_{\lambda_0}$ such that
 \begin{enumerate}
 \item[(i)] $X_{\lambda_0}({\mathscr K}_{i,j}(\mathcal{I}))(\lambda_0)= ({\mathscr K}_{i,j}(\mathcal{J}))(\lambda_0))X_{\lambda_0}, 0\leq i,j$;
\item[(ii)] $ X_{\lambda_0} \partial^{i}\mathcal{I}(\lambda_0)\overline{\partial}^j\mathcal{I}(\lambda_0)= \partial^{i}\mathcal{J}(\lambda_0)\overline{\partial}^j\mathcal{J}(\lambda_0)X_{\lambda_0}, 0\leq
i,j;$
  \end{enumerate}

\end{thm}

\section{Cowen-Douglas operators }

We will discuss the relationship between  holomorphic curves in
$C^*$-algebras and holomorphic curves in Cowen-Douglas theory. We
will prove for any $T$  the projection-valued functions $P_T$ (with
$RanP_T(\lambda)=Ker(T-\lambda)$) will belongs to  ${\mathcal
A}_{n-1}(\Omega, {\mathcal L}({\mathcal H})).$  Secondly, we will give a similarity
theorem for Cowen-Douglas operators by using the curvature formulas of appropriate extended holomorphic curves.

\begin{defn} \cite{CD}
 Let $\Omega$ be a
bounded and connected open subset of the complex plane $\mathbb{C}$
and $n$ a positive integer. Let ${\mathcal B}_n(\Omega)$ denote the
set of operators $T$ in ${\mathcal L}({\mathcal H})$ satisfying:
$$
\begin{array}{cl}
(1)&{\Omega}{\subset}{\sigma}(T):=\{\lambda{\in}\mathbb{C},
T-\lambda \,\,
 {\mbox{is not invertible}} \}; \\
(2)&Ran(T-\lambda)={\mathcal H} \,\,\, {\mbox{for every}}\,\, \lambda{\in}{\Omega}; \\
(3)&\bigvee\limits_{\lambda{\in}{\Omega}}\{ker(T-\lambda):
\lambda{\in}{\Omega}\}={\mathcal
 H}; \,\, {\mbox{and}} \\
(4)&dimker(T-\lambda)=n \,\,\,{\mbox{for every}} \,\,
\lambda{\in}{\Omega}.
 \end{array}
 $$ We call an operator in
${\mathcal B}_n(\Omega)$ a Cowen-Douglas operator with index $n$.

Let $T$ be an operator in ${\mathcal B}_n(\Omega)$. Let $E_T$ denote
the Hermitian holomorphic vector bundle induced by $T$. In this
sense, we set $E_T(\lambda)=\mbox{Ker}(T-\lambda)$ and
$$\alpha=(\alpha_1,\alpha_2,\cdots,\alpha_n)$$
where $\{\alpha_i(\lambda)\}^{n}_{i=1}$ are the frames of
$E_T(\lambda)$ for any $\lambda\in \Omega.$

Following M. I. Cowen and R. G. Douglas, a curvature function for
$T\in B_n(\Omega)$ or $E_T$ can be defined as:

\end{defn}

\begin{defn}\label{curvature}\cite{CD}
$$K_T(\lambda)
=-\frac{\partial}{\partial \overline{\lambda}}(h^{-1}\frac{\partial
h}{\partial \lambda}), \mbox{for all}~ \lambda\in \Omega,$$

where the metric
$$h(\lambda)=(<\alpha_j(\lambda),\alpha_i(\lambda)>)_{n\times n},\forall
\lambda\in \Omega,$$ and
$\{\alpha_1(\lambda),\alpha_2(\lambda),\cdots,{\alpha}_n(\lambda)\}
$ are the frames of $E_T$. The partial derivatives of curvature are
defined as the following:

Let $E_T$ be a Hermitian holomorphic bundle induced by a Cowen-Douglas operator $T$, and $K_T$ be a curvature of $T$. Then we have that\\
 (1)\,\,$K_{T,\overline{z}}=\frac{\partial}{\overline{\partial}\lambda}(K_{T});$\\
 (2)\,\,$K_{T,z}=\frac{\partial}{\partial z}(K_T)+[h^{-1}\frac{\partial}{\partial z}h,K_{T}].$

By the definition above,  we can get the covariant  derivatives of
curvature: $K_{T,z^i\overline{z}^j}$, $i,j\in \mathbb{N}\cup \{0\}$
by using the inductive formulaes above.

\end{defn}

\begin{rem}
In the Definition \ref{curvature}, the curvature $K_T$ and the
covariant derivatives of curvature $K_{T,z^i\overline{z}^j}$ are the
matrices form of curvature and its covariant derivatives according
to frame $\alpha=\{\alpha_1,\alpha_2,\cdots,\alpha_n\}$. In
\cite{CD}, M. I .Cowen and R. G. Douglas use the notations
$K_T(\alpha)$ and $K_{T,z^i\overline{z}^j}(\alpha)$ and $K_T$ and
$K_{T,z^i\overline{z}^j}$ are regarded as bundle maps on $E_T$.
\end{rem}

\begin{defn}\label{PT}
 Let $T\in {\mathcal L({\mathcal H})}$ be a Cowen-Douglas operator.
 Suppose that $\{\alpha_i\}^{n}_{i=1}$ are the frames of
$E_T$ for any $\lambda\in \Omega.$  Let
$\alpha=(\alpha_1,\alpha_2,\cdots,\alpha_n)$, and $h$ denote the
metric of $E_T$ induced by $\alpha$.

Since ${\mathcal H}$ is a separable Hilbert space, for any fixed
$\lambda\in \Omega$ and $0\leq i\leq n$, $\alpha_i(\lambda)$ can be
regarded as an element of $l^2$. So in the following part, we always
will not  discriminate $\alpha_i$ and its coordinate in $l^2$. For
example, just like in Hardy space, we will always regard the kernel
function $(1-zw)^{-1}$ as same to the element $(1,w, w^2, \cdots,
w^n\cdots )\in l^2, z, w\in \mathbb{D}$.

  In the Definition \ref{holomorphic}, if
we choose the C*-algebra $B$ equal to $\mathbb{C}$, then ${\mathcal
U}:={\mathscr L}(l^2(\mathbb{N},B))={\mathcal L({\mathcal H})}$.
Note that 
 $$\begin{array}{llll}
 h(\lambda)&:=&((<\alpha_j(\lambda),\alpha_i(\lambda)>))_{i,j=1}^{n}\\
  &=&\alpha^{*}(\lambda)\cdot\alpha(\lambda).
 \end{array}$$ where $h$ is the metric of holomorphic bundle $E_T$. 

Define $P_T: \Omega\rightarrow {\mathcal P({\mathcal L({\mathcal
H})})}$ as follows:
$$P_T(\lambda):=\alpha(\lambda)\cdot(h(\lambda))^{-1}\cdot\alpha^{*}(\lambda) $$

 Then  $P_{T}:\Omega\rightarrow {\mathcal P({\mathcal L({\mathcal
H})})}$ is also an  holomorphic curve. For any $\lambda\in
\Omega$, by a direct computation as we mentioned in chapter two,
$P_{T}(\lambda)$ is the matrix form of the projection from
${\mathcal H}$ to $\mbox{Ker}(T-\lambda)$ as an operator, In the
following, we will also not discriminate them.
\end{defn}

In order to prove the two theorems above, we need the following
notations and lemmas.

\begin{defn}\cite{CD}
Let E be a  Hermitian holomorphic vector bundle of rank $n$ over
$\Omega$ with metric-preserving connection D and curvature $K_{E}$,
(or $K$, for simplicity).  For any $\lambda\in \Omega$, let
$\mathscr{A}^{K}(\lambda)$ denote the algebra generated by the
curvatures and their covariant derivatives at $\lambda$. Choosing a
particular frame $\alpha$, we use symbol
$\mathscr{A}^{K}(S)(\lambda)$ denote the the matrix algebra
generated by the matrices of the curvatures and their covariant
derivatives at $\lambda$.
 Let $\tau\in \mathbb{Z}_{+}$,
Let $\mathscr{A}^{K}_{\tau}(\lambda)$ denote the algebra generated
by the covariant derivatives of the curvature of order at most
$\tau$.
\end{defn}

\begin{defn}\cite{CD}
The $j$th coalescing set for the curvature, denoted $\mathscr{C}_j$,
is the set of all $\lambda$ in $\Omega$ such that
dim$\mathscr{A}^{K}_{\tau}(\lambda)$ fails to be locally constant
for at least some $i,0<i<\tau.$ And the $\tau$ th coalescing set is
closed and nowhere dense in $\Omega$.
\end{defn}

\begin{defn}\cite{CD}
The generating order of the connection $D$, denoted by
$g(D,\Omega)$, to be the smallest integer $\tau$ such that
$\mathscr{A}^{K}(\lambda)$  is generated by covariant derivatives of
the curvature of total order at most $\tau$, for all $\lambda$ in
$\Omega\cap {\mathscr{C}_{\tau}}^c$.
\end{defn}

\begin{rem}
Let $s$ and $\tilde{s}$ be two frame of $E$ and $s=\tilde{s}A$, $A$
is an invertible holomorphic matrix-valued function. Then
$K_{z^i\overline{z}^j}(s)=AK_{z^i\overline{z}^j}(\tilde{s})A^{-1}$.
Thus we can know that the  generating order does not depend the the
frame $s$.
\end{rem}

\begin{lem}\label{MnC}\cite{CD}
Let E be a $C^{\infty}$ Hermitian vector bundle of dimension $n$
over an open subset $\Omega$ in $\mathbb{C}^k$, with
metric-preserving connection $D$. Let $\lambda_0$ be in $\Omega$,
$\lambda_0$ not in the coalescing set for the curvature. Then there
exists a neighborhood $\Omega_0$, of $z_0$ in $\Omega$ and a
$C^\infty$ orthonormal frame $s$ for $E$ over $\Omega$, with the
properties: and $\mathscr{A}^{K}(S)(\lambda)=M({\mathscr{N}},\otimes
{\mathscr{M}})$ for all $\lambda$ in $\Omega$. And the generating
order is less than $n-1$.

\end{lem}

\begin{thm}\label{pnu}
Let $f:\Omega\rightarrow Gr(n,{\mathcal H})$ be a holomorphic curve
with
$$f(\lambda)=\bigvee\{\alpha_1(\lambda),\alpha_2(\lambda),\cdots,\alpha_n(\lambda)\},
\forall \lambda\in \Omega.$$ Let $E_f$ be the pull back bundle
induced by $f$ with metric-preserving connection. Set
$\alpha=(\alpha_1(\lambda),\alpha_2(\lambda),\cdots,\alpha_n(\lambda)),
h=\alpha^*\alpha$. Let $P$ be the  holomorphic curves
defined in Definition 2.1 according to $\alpha$ and $h, i=1,2$
respectively. If $\bigvee\limits_{\lambda\in \Omega}
f(\lambda)={\mathcal H}$,  then $P\in {\mathcal A}_{n-1}(\Omega,
{\mathcal L}({\mathcal H})). $

\end{thm}

\begin{proof}
For each $\lambda\in \Omega$ and every $\tau\in \mathbb{Z}_{+}\cup
\{\infty\}$,
 set
 $${\mathcal F}^{\tau}_{\lambda}=\{{\mathscr K}_{i,j}(P)(\lambda):i,j\in \mathbb{Z}_{+}, i,j\leq \tau\}.$$
 And let
${\mathscr F}^{\tau}_{\lambda}$ be the closure of $*$-subalgebra of
$\mathcal L({\mathcal H})$ generated by ${\mathcal
F}^{\tau}_{\lambda}$, where ${\mathscr K}_{i,j}(P)$ is the covariant
derivative with the property
$${\mathscr K}_{i,j}(P)=\alpha(-K_{P,z^i,{\overline
 z}^j})h_1^{-1}\alpha^*.$$
Recall that $${\mathcal
B}^{\tau}_{\lambda}=\{\overline{\partial}^{J}P(\lambda)\partial^{I}P(\lambda):I,J\in
\mathbb{Z}_{+}, I,J\leq \tau\}.$$ And ${\mathcal
U}^{\tau}_{\lambda}$ be the closure of $*$-subalgebra of $\mathcal
U$ generated by ${\mathcal B}^{\tau}_{\lambda}$ (See in 1.2). Note
that when $i,j\leq 2,$ we have
$${\mathscr K}_{0,0}(P)=\overline{\partial}P\partial P,
 {\mathscr K}_{1,0}(P)=\overline{\partial}^2P\partial P,
{\mathscr K}_{0,1}(P)=\overline{\partial}P\partial^2 P, {\mathscr
K}_{1,1}(P)=\overline{\partial}^2P\partial^2
 P-2(\overline{\partial}P\partial P)^2.$$
By the inductive formulaes for all ${\mathscr K}_{i,j}(P),i,j\in
\mathbb{Z}_{+}$ (See Definition 2.1), $${\mathscr
K}_{i+1,j}(P)=P(\partial ({\mathscr K}_{i,j}(P))), {\mathscr
K}_{i,j+1}(P)=(\overline{\partial}({\mathscr K}_{i,j}(P)))P,$$ and
the same proof of Lemma \ref{tran}, we can prove that ${\mathcal
F}^{\tau}_{\lambda}={\mathcal B}^{\tau}_{\lambda}, \forall \tau\in
\mathbb{Z}_{+}$, and ${\mathscr F}^{\tau}_{\lambda}={\mathcal
U}^{\tau}_{\lambda}, \forall \tau\in \mathbb{Z}_{+}$. In special,
${\mathscr F}^{\infty}_{\lambda}={\mathcal U}^{\infty}_{\lambda},
\forall \tau\in \mathbb{Z}_{+}, \forall \lambda\in\Omega.$

By the covariant curvature formulae ${\mathscr
K}_{i,j}(P)=\alpha(-K_{P,z^i,{\overline
 z}^j})h_1^{-1}\alpha^*=\alpha(-K_{E,z^i,{\overline
 z}^j})h_1^{-1}\alpha^*$, ${\mathscr
F}^{\tau}_{\lambda}\cong \mathscr{A}^{K}_{\tau}(\lambda)$. By Lemma
\ref{MnC}, we know that for any $\lambda\in \Omega$ minus the
coalescing set, ${\mathcal U}^{\infty}_{\lambda}\cong
M({\mathscr{N}},\otimes {\mathscr{M}})$ and generating order
according to
$\overline{\partial}^{J}P(\lambda)\partial^{I}P(\lambda)$ is no more
than $n-1.$ Thus, the conditions (1) and (2) of Definition
\ref{Uinfi} follows. Since $\bigvee\limits_{\lambda\in \Omega}
f(\lambda)={\mathcal H}$, then any operator $a$ in ${\mathcal
L}({\mathcal H})$ with $aP(\lambda)=0, \forall \lambda\in \Omega$
would be equal to zero.
\end{proof}

 By Theorem \ref{pnu}, we can see that Theorem
\ref{MSlemma} is generalization of the unitary classification
theorem of Cowen-Douglas operators in\cite{CD}. One also can give
another proof  by using By Theorem \ref{pnu} and Theorem
\ref{MSlemma}. And the $\overline{\partial}^J P\partial^I P, I,J\leq
n$ mentioned in Theorem \ref{MSlemma} are essentially the curvature and
covariant derivatives of corresponding Cowen-Douglas operators. In
this sense, we can also regards the following curvature formulaes
mentioned in Definition 2.1:
$${\mathscr K}_{i+1,j}(P)=P(\partial ({\mathscr K}_{i,j}(P))),
{\mathscr K}_{i,j+1}(P)=(\overline{\partial}({\mathscr
K}_{i,j}(P)))P,$$as the curvature formulaes of  holomorphic
curves in general C*-algebras.

\subsection{Similarity of Cowen-Douglas operators and extended
holomorphic curves}

\begin{defn}\label{IT}
 Let $T\in B_n(\Omega)$ be a Cowen-Douglas operator.
 Suppose that $\{\alpha_i\}^{n}_{i=1}$ are the frames of
$E_T$ for any $\lambda\in \Omega.$  Let
$\alpha=(\alpha_1,\alpha_2,\cdots,\alpha_n)$, and $h$ denote the
metric of $E_T$ induced by $\alpha$. As we point out in Definition
\ref{PT}, we always will not  discriminate $\alpha_i$ and its
coordinate in $l^2$.

Choosing $B=\mathbb{C}$ in  Definition \ref{Ide}, now define
${\mathcal I}_T: \Omega\rightarrow {\mathcal I}(\mathcal
{L}(\mathcal {H}))$ as follows:
$${\mathcal I}_T(\lambda):=\alpha(\lambda)\cdot(\beta^{*}(\lambda)\alpha(\lambda))^{-1}\cdot\beta^{*}(\lambda)$$
 Then  ${\mathcal I}_T: \Omega\rightarrow  {\mathcal I}(\mathcal
{L}(\mathcal {H}))$ is also an extended holomorphic curve. For any
$\lambda\in \Omega$, similar to Definition \ref{PT}, ${\mathcal
I}_T(\lambda)$ is the matrix form of the idempotent from ${\mathcal
H}$ to $\mbox{Ker}(T-\lambda)$ as an operator, In the following, we
will also not discriminate them.
\end{defn}

By Theorem \ref{mainthm3} and using the extended curvature and
covariant derivatives given in Definition \ref{GECur} , we have the
following proposition about the similarity of extended holomorphic
curves induced by Cowen-Douglas operators.
\begin{prop} Let $T_1,T_2\in B_n(\Omega)$. Let ${\mathcal
I}_{T_1}$ and ${\mathcal I}_{T_2}$ according to $T_i$, $i=1,2$  be
two extended holomorphic curves given by Definition \ref{IT}.
 Then  ${\mathcal
I}_{T_1}\sim_s {\mathcal I}_{T_2}$ if and only if there exists
$\lambda_0\in
 \Omega$ and invertible operator $X_{\lambda_0}$ such that
 \begin{enumerate}
 \item[(i)] $X_{\lambda_0}({\mathscr K}_{i,j}(\mathcal{I}_{T_1}))(\lambda_0)= ({\mathscr K}_{i,j}(\mathcal{I}_{T_2}))(\lambda_0))X_{\lambda_0}, 0\leq i,j$;
\item[(ii)] $ X_{\lambda_0} \partial^{i}\mathcal{I}_{T_1}(\lambda_0)\overline{\partial}^j\mathcal{I}_{T_1}(\lambda_0)
=\partial^{i}\mathcal{I}_{T_2}(\lambda_0)\overline{\partial}^j\mathcal{I}_{T_2}(\lambda_0)X_{\lambda_0},
0\leq i,j;$
  \end{enumerate}
 \end{prop}

Since ${\mathcal I}_{T_1}(\lambda)$ and ${\mathcal
I}_{T_2}(\lambda)$ are idempotents on to $\mbox{Ker}(T_i-\lambda),
\lambda\in \Omega$ respectively. Then ${\mathcal I}_{T_1}\sim_s
{\mathcal I}_{T_2}$ will implies the similarity of $T_1$ and $T_2$.
And on the other hand, when $T_1$ and $T_2$ are similar, one also
can find two extended holomorphic curves ${\mathcal I}_{T_1}$ and
${\mathcal I}_{T_2}$ are similarity equivalent.  Thus, we have a
sufficient an necessary condition which involving the extended
curvature for the similarity of any two Cowen-Douglas operators with
high rank as the following:

\begin{prop} Let $T_1,T_2\in B_n(\Omega)$.   Then $T_1\sim T_2$
if and only if  there exists two extended holomorphic curves
${\mathcal I}_{T_1}$ and ${\mathcal I}_{T_2}$ according to $T_i$,
$i=1,2$ which are defined in Definition \ref{IT} and an invertible
operator $X_{\lambda_0}$ such that the following statements hold for
some fixed $\lambda_0\in \Omega$,
 \begin{enumerate}
 \item[(1)] $X_{\lambda_0}({\mathscr K}_{i,j}(\mathcal{I}_{T_1}))(\lambda_0)= ({\mathscr K}_{i,j}(\mathcal{I}_{T_2}))(\lambda_0))X_{\lambda_0}, 0\leq i,j$;
\item[(2)] $ X_{\lambda_0} \partial^{i}\mathcal{I}_{T_1}(\lambda_0)\overline{\partial}^j\mathcal{I}_{T_1}(\lambda_0)
=\partial^{i}\mathcal{I}_{T_2}(\lambda_0)\overline{\partial}^j\mathcal{I}_{T_2}(\lambda_0)X_{\lambda_0},
0\leq i,j,$
  \end{enumerate}

 \end{prop}

\begin{rem}
Let $T,S\in B_1(\mathbb{D})$. Suppose that 
$$ker(T-\lambda)=\bigvee\{\alpha_0,\alpha_1\lambda, \cdots,\alpha_n\lambda^n,\cdots\},ker(S-\lambda)=\bigvee\{\beta_0,\beta_1\lambda, \cdots,\beta_n\lambda^n,\cdots\}$$
Set $$\alpha(\lambda)=\{\alpha_0,\alpha_1\lambda, \cdots,\alpha_n\lambda^n,\cdots\}^T,$$ 
$$\beta(\lambda)=\{\beta_0,\beta_1\lambda, \cdots,\beta_n\lambda^n,\cdots\}^T,$$ 
$$\gamma(\lambda)=\{\frac{\beta^2_0}{\alpha_0},\frac{\beta^2_1}{\alpha_1}\lambda, \cdots,\frac{\beta^2_n}{\alpha_n}\lambda^n,\cdots\}^T,$$ 
and 
$h(\lambda)=\sum\limits_{i=0}^{\infty}\beta_i^2(|\lambda|^2)^i.$

Now we choose the canonical extended holomorphic curves ${\mathcal I}_S , {\mathcal I}_T$ corresponding to $S$ and $T$ as the follows: 
$${\mathcal I}_S=\beta h^{-1} \beta^*, {\mathcal I}_T=\alpha h^{-1}\gamma^* $$

In the following,  we will point out the relationship between the similarity of ${\mathcal I}_S$ and ${\mathcal I}_T$ and  the result of A. L. Shields.

Firstly, by a direct computation, we have the following claim: 

{\bf{Claim}}\,\, Suppose that $P_1=\alpha h^{-1} \beta^*, P_2=\tilde{\alpha}h^{-1}\tilde{\beta}^*$ be two extended holomorphic curves and $h=\beta^*\alpha=\tilde{\beta}\tilde{\alpha}$. 
If there exists an invertible operator $X$ such that
$$X\partial^{i}P_1\overline{\partial}^jP_1
=\partial^{i}P_2\overline{\partial}^jP_2X,$$ then $X\partial^{i}\alpha\overline{\partial}^j\beta^*
=\partial^{i}\tilde{\alpha}\overline{\partial}^j\tilde{\beta}^*X$

Now suppose that there exists an invertible operator $X$ such that  $$X\partial^{i}{\mathcal I}_T\overline{\partial}^j{\mathcal I}_T
=\partial^{i}{\mathcal I}_S\overline{\partial}^j{\mathcal I}_SX, \forall i,j \in {\mathbb{N}}.$$  Note that $\gamma^*\alpha=\beta^*\beta=h$. Then we have  
$$X\partial^{i}\alpha\overline{\partial}^j\gamma^*
=\partial^{i}\beta\overline{\partial}^j\beta^*X$$

Choosing $\lambda=0$, then we have 
$$\partial^{i}\alpha(0)\overline{\partial}^j\gamma^*(0)=\frac{\alpha_i\beta^2_{j}}{\alpha_{j}}i!j!e_{i+1,j+1},  \forall i,j \in {\mathbb{N}} $$
and 
$$\partial^{i}\beta(0)\overline{\partial}^j\beta^*(0)=\beta_i\beta_ji!j!e_{i+1,j+1},  \forall i,j \in {\mathbb{N}}$$
where $e_{i,j}$ denote the infinite matrix which satisfies  that $(i,j)^{th}$ entry equals to one and other entries are all zero.

Set $((x_{i,j}))$ to be the matrix form of $X$ . Without loss of generality, we can  assume $x_{1,1}$ is not equal to zero.   Then we have 
$$((x_{i,j}))\frac{\alpha_i\beta^2_{j}}{\alpha_{j}}e_{i+1,j+1}=\beta_i\beta_je_{i+1,j+1}((x_{i,j})), \forall i,j \in {\mathbb{N}}$$
Choosing $j=i+1$, we have $x_{n,n}=\frac{\beta_{n-1}\alpha_0}{\alpha_{n-1}\alpha_0}x_{1,1}$. By the result of A. L .Shields, we can see if $T$ is not similar to 
$S$, $\frac{\beta_{n}}{\alpha_{n}}$ will goes to $\infty$ or zero, when $n\rightarrow \infty$. That means $X$ could not be invertible. There is an contradiction. 

\end{rem}

\section{Application and Examples}

\subsection{Trace of The Derivatives of Curvature }

In \cite{Treil}, S. Treil and B. D. Wick gave a sufficient condition
for the existence of a bounded analytic projection onto a
holomorphic family of generally infinite dimensional subspaces
induced by some holomorphic bundles. As the corollaries of this,
they also obtained some new results about the Operator Corona
Problem.

Let $E$ be a Hilbert space and $P:\mathbb{D}\rightarrow B(E)$ be an
$C^2$, projection valued function and $P\partial P=0$. In
\cite{Treil}, as the main theorem,  it was proved that if there
exists a bounded non-negative subharmonic function $\psi$ such that
$$\triangle \psi(\lambda)\geq \|\partial P(\lambda)\|^2, \forall \lambda\in \mathbb{D},$$
then there exists some function $V\in L^{\infty}_{E\rightarrow E}$
such that $P-(I-P)VP\in H^{\infty}_{E\rightarrow E}.$

As an application of this result to the similarity of Cowen-Douglas
operators, H.Kwon and S.Treil gave the following theorem to decide
when a contraction operator $T$ will be similar to the $n$ times
copies of $S^*_z$ on Hardy space \cite{Kwon1}.

\begin{thm}\label{Kwon}\cite{Kwon1} {\it Let $\mathbb{D}$ be the unit disk and $T\in
B_n(\mathbb{D})$ with $||T||\leq 1$. For any $\lambda\in\mathbb{D}$,
let $P_T(\lambda)$ be the orthogonal projection onto
$\mbox{Ker}(T-\lambda)$. Then $T$ is similar to the backward shift
operator $S^{*}_n$ if and only if there exists a bounded subharmonic
function $\psi$ such that
$$||\frac{\partial P_T(\lambda)}{\partial
\lambda}||^2_{HS}-\frac{n}{(1-|\lambda|^2)^2}\leq \Delta
\psi(\lambda),\forall \lambda\in \mathbb{D},$$ where $||.||_{HS}$
denotes the Hilbert-Schmidt norm.}
\end{thm}

And this result was also generalized to a general analytic
functional  space  ${\mathcal M_n}$ (see more details in
\cite{Kwon2}) by R. G. Douglas, H. Kwon and S. Treil. If we also
only consider Cowen-Douglas operator class, then we have that

 \begin{thm}\label{Kwon2}\cite{Kwon2}  Let $T\in
B_m(\mathbb{D})$ be an n-hypercontraction. Then $T$ is similar to
the backward shift operator $S^*_{n,C^m}$ if and only if there
exists a bounded subharmonic function $\psi$ such that
$$||\frac{\partial
P_T(\lambda)}{\partial
\lambda}||^2_{HS}-\frac{nm}{(1-|\lambda|^2)^2}= \Delta
\psi(\lambda),\forall \lambda\in \mathbb{D}.$$ \end{thm}

 In \cite{Kwon1}, $-||\frac{\partial P_T(\lambda)}{\partial \lambda}||^2_{HS}$ is
pointed out to be  curvature of the eigenvector bundle
$\mbox{Ker}(T-\lambda I)$ and Hardy shift case of this claim was
also proved in \cite{Kwon1}. And a  proof of $B_1(\Omega)$ case was
given by J. Sarkar in \cite{Sarkar}.

For any given Cowen-Douglas operator $T$, the second author joint
with Y. Hou and H. Kwon proved  that $||\frac{\partial P_T}{\partial
\lambda}||^2_{HS}$ and curvature $K_T$ have the following
relationship:

\begin{prop}\label{traceK}\cite{HJK} For any operator $T\in B_n(\Omega)$, trace
$K_T(\lambda)=-||\partial P_T(\lambda)||^2_{{\mathcal
 HS}}$, $\forall \lambda\in \Omega$.
 \end{prop}

By Lemma \ref{main}, we also can generalize the Proposition
\ref{traceK} to the case of covariant derivatives to give the
description of the trace.

\begin{prop}\label{corvariantf} Let $T\in B_n(\Omega)$ and
 $P_T:\Omega\rightarrow {\mathcal L}({\mathcal H})$ be an
 holomorphic curve with $P_T(\lambda)$ is the projection induced by
 $\mbox{Ker}(T-\lambda)$ for any $\lambda\in \Omega$.  Let $\{\sigma_i\}^{\infty}_{i=1}\}$ be the orthogonal normalize bases of ${\mathcal H}$.
  Then for any $s,t\leq n$, we have
 $$\sum\limits_{i=1}^{\infty}<{\mathscr K}_{s,t}(P_T)\sigma_i, \sigma_i>=-trace(K_{T,z^s,\overline{z}^t}).$$
 \end{prop}
\begin{proof}
Let $e_i=(0,0,\cdots,0,1,0\cdots,0)^T$ be the coordinate of
$\sigma_i$ By Lemma 2.6, we have
$${\mathscr K}_{s,t}(P_T)=-\alpha K_{T,z^s,\overline{z}^t} h^{-1} \alpha^*,
s,t\leq n.$$ Then it follows that for any $i$, we have

\begin{eqnarray*}\begin{array}{lllll} <{\mathscr
K}_{s,t}(P_T)\sigma_i, \sigma_i>&=&-<\alpha K_{T,z^s,\overline{z}^t}
h^{-1} \alpha^*e_i, e_i>\\
&=&-<K_{T,z^s,\overline{z}^t} h^{-1} \alpha^*e_i, \alpha ^*e_i>\\
&=&-<K_{T,z^s,\overline{z}^t} h^{-1} \left(
                              \begin{array}{llllll}\alpha^{1*}_1
\alpha^{*2}_1 \cdots  \alpha^{l*}_1 \cdots\\
\alpha^{1*}_2
\alpha^{*2}_2 \cdots  \alpha^{l*}_2 \cdots\\
\cdots\cdots\cdots\cdots\cdots\\
\alpha^{1*}_n \alpha^{*2}_n \cdots  \alpha^{l*}_n \cdots\\
\end{array}\right)\left(\begin{array}{llll}0\\\vdots\\1\\\vdots\end{array}\right), \left(
                              \begin{array}{llllll}\alpha^{1*}_1
\alpha^{*2}_1 \cdots  \alpha^{l*}_1 \cdots\\
\alpha^{1*}_2
\alpha^{*2}_2 \cdots  \alpha^{l*}_2 \cdots\\
\cdots\cdots\cdots\cdots\cdots\\
\alpha^{1*}_n \alpha^{*2}_n \cdots  \alpha^{l*}_n \cdots\\
\end{array}\right)\left(\begin{array}{llll}0\\\vdots\\1\\\vdots\end{array}\right)>\\
&=&-<K_{T,z^s,\overline{z}^t} h^{-1}
\left(\begin{array}{llll}\alpha^{i*}_1\\
\alpha^{i*}_2\\\vdots\\ \alpha^{i^*}_n\\\end{array}\right),
\left(\begin{array}{llll}\alpha^{i*}_1\\
\alpha^{i*}_2\\\vdots\\ \alpha^{i^*}_n\\\end{array}\right)>\\
\end{array}\end{eqnarray*}
Now let
\begin{eqnarray*}h=\begin{pmatrix}
h_{11}& h_{12}&\cdots&h_{1n}\\
h_{21}& h_{22}& \cdots&h_{2n}\\
\cdots& \cdots& \cdots&\cdots\\
h_{n1}&h_{n2}& \cdots&h_{nn}\\
\end{pmatrix}, h^{-1}=\begin{pmatrix}
\widetilde{h}_{11}& \widetilde{h}_{12}&\cdots&\widetilde{h}_{1n}\\
\widetilde{h}_{21}& \widetilde{h}_{22}& \cdots&\widetilde{h}_{2n}\\
\cdots& \cdots& \cdots&\cdots\\
\widetilde{h}_{n1}&\widetilde{h}_{n2}& \cdots&\widetilde{h}_{nn}\\
\end{pmatrix},\end{eqnarray*}
and
\begin{eqnarray*}-K_{T,z^s,\overline{z}^t}=\begin{pmatrix}
K_{11}& K_{12}&\cdots&K_{1n}\\
K_{21}& K_{22}& \cdots&K_{2n}\\
\cdots& \cdots& \cdots&\cdots\\
K_{n1}&K_{n2}& \cdots&K_{nn}\\
\end{pmatrix}.\end{eqnarray*}
 Then we have
\begin{eqnarray*}\begin{array}{lllll}
-<K_{T,z^s,\overline{z}^t} h^{-1}
\left(\begin{array}{llll}\alpha^{i*}_1\\
\alpha^{i*}_2\\ \vdots\\ \alpha^{i^*}_n\\\end{array}\right),
\left(\begin{array}{llll}\alpha^{i*}_1\\
\alpha^{i*}_2\\ \vdots\\ \alpha^{i^*}_n\\\end{array}\right)>
&=&(\alpha^i_1,\alpha^i_2,\cdots,\alpha^i_n)K_{T,z^s,\overline{z}^t} h^{-1}\left(\begin{array}{llll}\alpha^{i*}_1\\
\alpha^{i*}_2\\ \vdots\\ \alpha^{i^*}_n\\\end{array}\right)\\
\end{array}\end{eqnarray*}
\begin{eqnarray*}\begin{array}{lllll}
&=&(\alpha^i_1,\alpha^i_2,\cdots,\alpha^i_n)\begin{pmatrix}
K_{11}& K_{12}&\cdots&K_{1n}\\
K_{21}& K_{22}& \cdots&K_{2n}\\
\cdots& \cdots& \cdots&\cdots\\
K_{n1}&K_{n2}& \cdots&K_{nn}\\
\end{pmatrix}h^{-1}\left(\begin{array}{llll}\alpha^{i*}_1\\
\alpha^{i*}_2\\ \vdots\\ \alpha^{i^*}_n\\\end{array}\right)\\
\end{array}\end{eqnarray*}
\begin{eqnarray*}\begin{array}{lllll}
&=&(\sum\limits_{j=1}^nK_{j1}\alpha^i_j,\cdots,\sum\limits_{j=1}^nK_{jn}\alpha^i_j)\left(\begin{array}{llll}\sum\limits_{k=1}^n\widetilde{h}_{1,k}\alpha^{i*}_k\\
\sum\limits_{k=1}^n\widetilde{h}_{2,k}\alpha^{i*}_k\\ \vdots\\ \sum\limits_{k=1}^n\widetilde{h}_{n,k}\alpha^{i*}_k\\\end{array}\right)\\
&=&(\sum\limits_{j=1}^nK_{j1}\alpha^i_j)(\sum\limits_{k=1}^n\widetilde{h}_{1,k}\alpha^{i*}_k)+\cdots+(\sum\limits_{j=1}^nK_{jn}\alpha^i_j)(\sum\limits_{k=1}^n\widetilde{h}_{n,k}\alpha^{i*}_k)\\
&=&(\sum\limits_{j=1}^n\sum\limits_{k=1}^n
K_{j1}\widetilde{h}_{1,k}\alpha^i_j\alpha^{i*}_k)+\cdots+(\sum\limits_{j=1}^n\sum\limits_{k=1}^n
K_{jn}\widetilde{h}_{n,k}\alpha^i_j\alpha^{i*}_k).
\end{array}\end{eqnarray*}
Thus, we have that
$$\sum\limits_{i=1}^{\infty}<{\mathscr K}_{s,t}(P)\sigma_i, \sigma_i>=(\sum\limits_{i=1}^{\infty}\sum\limits_{j=1}^n\sum\limits_{k=1}^n
K_{j1}\widetilde{h}_{1,k}\alpha^i_j\alpha^{i*}_k)+\cdots+(\sum\limits_{i=1}^{\infty}\sum\limits_{j=1}^n\sum\limits_{k=1}^n
K_{jn}\widetilde{h}_{n,k}\alpha^i_j\alpha^{i*}_k)\eqno{(4.1.1)}$$

Since $h^{-1}h=I$, then it follows that
$$\sum\limits_{k=1}^n \widetilde{h}_{l,k}h_{k,j}=0, l\neq j;
\sum\limits_{k=1}^n \widetilde{h}_{l,k}h_{k,j}=1,l=j.$$ Note that
$h=\alpha\alpha^*$, then we have
$h_{k,j}=\sum\limits_{i=1}^{\infty}\alpha^i_j\alpha^{i*}_{k}.$ Then
we have
$$0=\sum\limits_{k=1}^n \widetilde{h}_{l,k}h_{k,j}=\sum\limits_{k=1}^n
\widetilde{h}_{l,k}(\sum\limits_{i=1}^{n}\alpha^i_j\alpha^{i*}_k),
j\neq l,$$ and
$$1=\sum\limits_{k=1}^n \widetilde{h}_{l,k}h_{k,j}=\sum\limits_{k=1}^n
\widetilde{h}_{l,k}(\sum\limits_{i=1}^{n}\alpha^i_j\alpha^{i*}_k),
j= l.$$

So for any $l=0,1,\cdots, n$, we have

$$\begin{array}{llll}\sum\limits_{i=1}^{\infty}\sum\limits_{j=1}^n\sum\limits_{k=1}^n
K_{j1}\widetilde{h}_{l,k}\alpha^i_j\alpha^{i*}_k&=&\sum\limits_{j=1}^nK_{jl}(\sum\limits_{k=1}^n\widetilde{h}_{l,k}\sum\limits_{i=1}^n\alpha^i_j\alpha^{i*}_k)\\
&=&\sum\limits_{j=1}^nK_{jl}(\sum\limits_{k=1}^n\widetilde{h}_{lk}h_{k,j})\\
&=&K_{ll}.
\end{array}$$

Thus, we have
$$ \sum\limits_{i=1}^{\infty}<{\mathscr K}_{s,t}(P_T)\sigma_i, \sigma_i>=\sum\limits_{l=1}^nK_{ll}=-trace(K_{T,z^s,\overline{z}^t}).$$
\end{proof}

Let $\mathscr{E}$ be a Hilbert space and $P:\mathbb{D}\rightarrow
B(\mathscr{E})$ be an $C^2$, projection valued function and
$P\partial P=0$. Note that $P$ is also a holomorphic curve on
$\Omega$.  Similar to the main theorem in \cite{Treil} which is
mentioned in the beginning of 4.1, a natural question is when one
can find $V_i\in L^{\infty}_{\mathscr{E} \rightarrow \mathscr{E}},
i=1,2$ such that
$$F:=\partial P-(I-P)V_1(I-P)-PV_2P\in H^{\infty}_{\mathscr{E}\rightarrow
\mathscr{E}}.$$ By a similar proof to Theorem 0.2 in \cite{Treil},
we can obtain the following lemma.

\begin{lem}
Let $\mathscr{E}$ be a Hilbert space and $P:\mathbb{D}\rightarrow
\mathcal {L}(\mathscr{E})$ be an $C^2$, projection valued function
and $P\partial P=0$. If there exists  bounded non-negative
subharmonic functions $\psi_i, i=1,2$ such that
$$\|\partial^i P(\lambda)\|^2 \leq \triangle \psi_i(\lambda), \forall \lambda\in \mathbb{D},$$
then there exists some function $V_i\in
L^{\infty}_{\mathscr{E}\rightarrow \mathscr{E}}, i=1,2$ such that
$\partial P-(I-P)V_1(I-P)-PV_2P\in
H^{\infty}_{\mathscr{E}\rightarrow \mathscr{E}}.$

\end{lem}

\begin{prop}
Let $T\in B_n(\mathbb{D})\cap {\mathcal L}({\mathcal
H}^2_{\mathscr{E}}) $ and
 $P_T:\mathbb{D}\rightarrow {\mathcal L}({\mathcal H}^2_{\mathscr{E}})$ be an
 holomorphic curve with $P_T(\lambda)$ is the projection from ${\mathcal H}^2_{\mathscr{E}}$ to
 $\mbox{Ker}(T-\lambda)$ for any $\lambda\in \mathbb{D}$.  If there exists  bounded non-negative
subharmonic functions $\psi_i, i=1,2$ such that
$$-(trace(K_T(\lambda))\leq \triangle \psi_1(\lambda),
 (trace(2K^2_T(\lambda))-trace(K_{T,z,\overline{z}}(\lambda))\leq \triangle \psi_2(\lambda), \forall \lambda\in \mathbb{D}, $$
then there exists some function $V_i\in
L^{\infty}_{\mathscr{E}\rightarrow \mathscr{E}}, i=1,2$ such that
$\partial P-(I-P)V_1(I-P)-PV_2P\in
H^{\infty}_{\mathscr{E}\rightarrow \mathscr{E}}.$
\end{prop}

\begin{proof}
Since we already know that $\|\partial
P(\lambda)\|^2=-(trace(K_T(\lambda))$, then we only need to prove
that $$\|\partial^2
P(\lambda)\|^2=(trace(2K^2_T(\lambda))-trace(K_{T,z,\overline{z}}(\lambda)),
\forall \lambda\in \mathbb{D}.$$

By formulae \ref{2.8}, we have ${\mathscr
K}_{1,1}(P)=\overline{\partial}^2P\partial^2
 P-2(\overline{\partial}P\partial P)^2.$ By Proposition \ref{corvariantf}, we
 have that  $$\begin{array}{llll}
 \sum\limits_{i=1}^{\infty}<{\mathscr K}_{1,1}(P_T)\sigma_i, \sigma_i>&=&\sum\limits_{i=1}^{\infty}<\overline{\partial}^2P\partial^2
 P\sigma_i, \sigma_i>-\sum\limits_{i=1}^{\infty}<2(\overline{\partial}P\partial P)^2\sigma_i, \sigma_i>\\
&=&\sum\limits_{i=1}^{\infty}<\partial^2
 P\sigma_i, \partial^2P\sigma_i>-\sum\limits_{i=1}^{\infty}<2(\overline{\partial}P\partial P)^2\sigma_i, \sigma_i>\\
&=&\|\partial^2P(\lambda)\|^2-\sum\limits_{i=1}^{\infty}<2\alpha K^2_{T} h^{-1} \alpha^*\sigma_i, \sigma_i>\\
&=&\|\partial^2P(\lambda)\|^2- 2trace(K^2_{T}(\lambda))\\
 \end{array}$$
Since $\sum\limits_{i=1}^{\infty}<{\mathscr K}_{1,1}(P_T)\sigma_i,
 \sigma_i>=-trace(K_{T,z,\overline{z}}(\lambda)), \forall \lambda\in
 \mathbb{D}.$ This finishes the proof of the Proposition.
\end{proof}

\subsection{Similarity of holomorphic bundles}
 Let $E$ be a Hermitian holomorphic vector bundle of rank $n$ over a
bounded domain $\Omega$ in $ \mathbb{C}$ and let $F$ be a sub-bundle
of $E$ of rank $r$. Suppose that $\{e_1, e_2, \cdots e_r\}$ is a
frame of $F$ and let $\{e_1, e_2, \cdots e_r,
e_{r+1},e_{r+2},\cdots, e_n\}$ be a frame of $E$ which is the
holomorphic frame according to the frame $\{e_1, e_2, \cdots e_r\}$ of
$F$. In this case, the quotient bundle $E/F$ possesses the frame
$\{[e_{r+1}],[e_{r+2}],\cdots, [e_n]\}$, where $[e_i],i=r+1,\cdots,
n$ are the equivalent class of $e_i$ in the quotient bundle $E/F$.

Set the metrics of $E$, $F$ and $E/F$ as the following:
$$h_E=((<e_j,e_i> ))^n_{i,j=1}, h_{F}=((<e_j,e_i> ))^r_{i,j=1},h_{E/F}=((<[e_j],[e_i]>
))^{n}_{i,j=r+1}.$$ Then we also can obtain the corresponding
curvatures of $E$, $F$ and $E/F$ as $K_{E}, K_{F}, K_{E/F}.$

Let $\mathscr{E},{\mathscr{E}}^*$ be two separable Hilbert spaces.
Let ${\mathscr{H}}_{{\mathscr{E}}^*\rightarrow \mathscr{E}}$ denote
the operator Hardy class of bounded analytic functions whose values
are the operators belongs to $B({\mathscr{E}}^*, \mathscr{E})$.

\begin{defn}\label{hol}
Let $\mathscr{E}$ and $\mathscr{F}$ be two separable Hilbert spaces.
Let $f$ be a holomorphic frame over a bounded domain $\Omega$ in $
\mathbb{C}$ with $f(\lambda)=\bigvee\{\{e_1(\lambda), e_2(\lambda),
\cdots e_r(\lambda)\}$.  Let $g$ be a  holomorphic frame over a
bounded domain $\Omega$ in $ \mathbb{C}$ with
$g(\lambda)=\bigvee\{e_1, e_2, \cdots e_r, e_{r+1},e_{r+2},\cdots,
e_n\}$. Suppose that $\bigvee\limits_{\lambda\in
\Omega}\{e_1(\lambda), e_2(\lambda), \cdots
e_r(\lambda)\}=\mathscr{F},$ and $\bigvee\limits_{\lambda\in
\Omega}\{e_1(\lambda), e_2(\lambda), \cdots
e_n(\lambda)\}=\mathscr{E}.$

Let $E_g/E_f$ be the quotient bundle induced by $f$ and $g$. Then
$E_g/E_f$ admits the following frame
$$E_g/E_f(\lambda)=\{[e_{r+1}],[e_{r+2}],\cdots, [e_n]\}.$$

\end{defn}

\begin{thm}\label{Treil}\cite{Treil} Let $\mathscr{E}$ be a Hilbert space and
$\Pi:D\rightarrow B(\mathscr{E})$ be an $C^2$ function which
satisfies that $\Pi^2=\Pi=\Pi^*$ and $\Pi\partial \Pi=0$. If there
exists a bounded non-negative subharmonic function $\psi$ such that
$$\triangle \psi(z)\geq \|\partial \Pi(z)\|_{{\mathcal HS}}^2, \forall z\in \mathbb{D}.$$
Then there exists a bounded analytic projection on to
$\mbox{Ran}\Pi(z)$, i.e. such that $P(z)$ is a projection onto
$\mbox{Ran}\Pi(z)$ for all $z\in \mathbb{D}$.
\end{thm}

\begin{lem}\label{Dinesh}\cite{Dinesh}
Let $0\rightarrow F\rightarrow E\rightarrow E/F$  be an exact
sequence of Hermitian holomorphic vector bundles. Then
$trace(K_{E/F}) = trace(K_E)-trace(K_F).$
\end{lem}

By using the curvature formulae in Lemma 2.6 (See (2.6.26)), Theorem
\ref{Treil}, Lemma \ref{Dinesh} and similar techniques in the main
theorems in \cite{Kwon1} and \cite{Kwon2}, we have the following
theorem:

\begin{thm}
 Let $f$ and $g$ be the holomorphic curves defined in Definition \ref{hol}.  Then $E_f\otimes
(E_g/E_f)\sim_s \bigoplus\limits_{i=1}^{n-r}E_f$ if and only if
 there exists a bounded harmonic function $\psi$ such that
$$trace K_{E_g}-trace{K_{E_f}}\leq \Delta \psi,$$ where
$\Delta$ denote the normalized Laplacian i.e. $\Delta=\partial
\overline{\partial}.$

\end{thm}

\end{document}